\documentclass[11pt, a4paper, noindent]{article} 

\usepackage[utf8]{inputenc} 

\usepackage{geometry} 
\usepackage{color,graphicx} 

\usepackage{amsmath}
\usepackage{amssymb}
\usepackage{amsthm}
\usepackage{chemarr}
\usepackage{esint}
\usepackage{cite}
\usepackage{booktabs} 
\usepackage{array} 
\usepackage{paralist} 
\usepackage{verbatim} 
\usepackage{subfig} 
\usepackage{dsfont}
\usepackage{mathtools}

\usepackage[yyyymmdd,hhmmss]{datetime}

\usepackage[nottoc,notlof,notlot]{tocbibind} 
\usepackage[titles,subfigure]{tocloft} 
\usepackage{xcolor}

\newtheorem{defi}{Definition}[section]
\newtheorem{prop}[defi]{Proposition}
\newtheorem{lma}[defi]{Lemma}
\newtheorem{rem}[defi]{Remark}

\newtheorem{thm}[defi]{Theorem}

\newtheorem{eg}[defi]{Example}
\DeclareMathOperator{\dist}{dist}
\DeclareMathOperator{\divergence}{div}

\DeclareMathOperator{\supp}{supp}

\definecolor{jan}{rgb}{0.7,0.1,0.1}

\def\N{{\mathbb N}} 
\def\Z{{\mathbb Z}} 
\def\R{{\mathbb R}} 
\def\C{{\mathcal C}}
\def\T{{\mathbb T}} 
\def\P{{\mathcal P}} 

\def\D{{\mathcal D}} 
\def\Hm{{\mathcal H}}
\def\Lm{{\mathcal L}}
\def\cT{{\mathcal T}}
\def\M {{\mathcal M}}
\def\CM {{\mathcal{CM}}}

\def\cA{{\mathcal A}}

\newcommand{\1}{\mathds{1}}

\newcommand{\eps}{\varepsilon}
\newcommand{\weakstar}{\stackrel{\ast}{\rightharpoonup}}
\newcommand{\loc}{\mathrm{loc}}

\title{Limits of density-constrained optimal transport}
\author{Peter Gladbach and Eva Kopfer}
\date{} 

\begin{document}

\maketitle

\begin{abstract}
 We consider the problem of dynamic optimal transport with a density constraint. We derive variational limits in terms of $\Gamma$-convergence for two singular phenomena. First, for densities constrained near a hyperplane we recover the optimal flow through an infinitesimal permeable membrane. Second, for rapidly oscillating periodic constraints we obtain the optimal flow through a homogenized porous medium.
\end{abstract}

\section{Introduction}
Over the last few years, optimal transport has become a vibrant research area with many different applications. In particular, density-constrained flow problems have garnered significant interest starting with the seminal work of Ford and Fulkerson \cite{FordFulkerson}.

In recent years, the theory of constraints has been adapted to optimal transport, first as a static version in \cite{korman2015} and then as dynamic constraints in \cite{buttazzo2009} and \cite{cardaliaguet2016}.

The model we use is based on the dynamic formulation of the Kantorovich distance due to Benamou and Brenier \cite{benamou2000},
\begin{equation}\label{eq: benamou}
W_2^2(\rho_0,\rho_1) = \inf\left\{\int_0^1 \int_{\R^d} \left|\frac{dV_t}{d\rho_t}\right|^2\,d\rho_t\,dt\,:\,\partial_t \rho_t + \divergence V_t = 0\right\}.
\end{equation}

Here the infimum is taken over all curves of probability measures $(\rho_t)_{t\in[0,1]} \subset \P(\R^d)$ with fixed endpoints.

In this paper, we constrain the densities of all intermediate measures $\rho_t$ by some measurable maximal density $h:\Omega \to [0,\infty]$. In this article $\Omega$ is a manifold with boundary, typically $\R^d$  or the torus $\T^d \coloneqq \R^d/\Z^d$.

More precisely we consider first the space $\M_+(\Omega)$ of finite nonnegative Radon measures on $\Omega$ equipped with the weak-$\ast$ (also called narrow) topology in duality with $\C_b(\Omega)$.

By extension the space of weakly-$\ast$ continuous curves of finite nonnegative Radon measures is
\begin{equation}
\CM_+(\Omega) := \left\{(\rho_t)_{t\in[0,1]}\subset \M_+(\Omega)\,:\,  t\mapsto \rho_t\text{ is weakly-$\ast$ continuous with constant mass} \right\}.
\end{equation}

We study in this article the constrained transport functional $E_h:\CM_+(\Omega) \to [0,\infty]$,
\begin{equation}\label{eq: problem}
\begin{aligned}
E_h((\rho_t)_{t\in[0,1]})\coloneqq
\begin{cases}
 \inf\left\{\int_0^1\int_{\Omega}\left|\tfrac{dV_t}{d\rho_t}\right|^2\, d\rho_t\, dt: \partial_t\rho_t+\divergence V_t=0
\text{ in } \D'((0,1) \times \Omega)  \right\},\\
\text{ if }\rho_t(A) \leq \int_A h(x)\,dx\text{ for all }t\in[0,1], A\subseteq \Omega\text{ open.}\\
\infty\text{, otherwise.}
\end{cases} 
\end{aligned}
\end{equation}
Note that the constraint is closed under weak-$\ast$ convergence by the Portmanteau theorem.

For $h=\infty$ we recover the classical Benamou-Brenier formula. If $h\in L^1_{\loc}(\Omega)$ then every admissible $\rho$ is absolutely continuous with density $\frac{d\rho}{dx}\leq h(x)$ almost everywhere. The case $h=\mathds{1}_U$ for some nonconvex $U\subset \Omega$, e.g. an hourglass (see Figure \ref{fig:hourglass}), models optimal transport of an incompressible but sprayable fluid. This specific problem was treated in \cite{liu2016euler}, \cite{liu2016least}. If $U$ is convex, $W_2$-geodesics between two measures $\rho_0,\rho_1\leq \mathds{1}_U$ satisfy the density constraints. If $U$ is not convex, optimal curves under the constraint are not $W_2$-geodesics and interact with the constraint.

\begin{figure}[h]
\begin{center}
\includegraphics{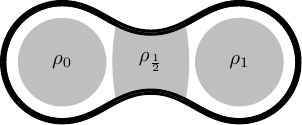}
\end{center}
\caption{Incompressible transport of mass through an hourglass. Here and in all figures, black denotes the mass exclusion region $\{h=0\}$, while white denotes the incompressible region $\{h=1\}$.}\label{fig:hourglass}
\end{figure}

We find the variational limits of two singular phenomena.

\subsection{Thin permeable membranes}

The first is the derivation of an infinitesimal membrane from the constraint
\begin{align}
h^\eps(x) \coloneqq
\begin{cases}
\alpha \eps,&\text{ if }x_d \in (0,\eps)\\
\infty,&\text{ otherwise.}
\end{cases}
\end{align}
for some $\alpha \in ( 0, \infty)$.

Then, as $\eps \to 0$, we derive an effective variational model in the sense of $\Gamma$-convergence as introduced by De Giorgi. We refer to \cite{braides} and \cite{dalmaso} for comprehensive overviews of the theory.

The limit functional acts on curves of nonnegative Radon measures on the topological disjoint union $\R^d_- \sqcup \R^d_+$ of the closed half-spaces
\[
\R^d_- := \R^{d-1} \times (-\infty,0]\text{ and }\R^d_+ := \R^{d-1}\times [0,\infty).
\]
and is given by $E_0:\CM_+(\R^d_- \sqcup \R^d_+) \to [0,\infty]$,
\begin{equation}\label{eq: membrane}
E_0((\rho_t^-,\rho_t^+)_{t\in [0,1]}) \coloneqq \inf\left\{ \int_0^1 \left(\int_{\R^d_-}\left|\frac{d V_t^-}{d\rho_t^-}\right|^2\,d\rho_t^- + \int_{\R^d_+}\left|\frac{d V_t^+}{d\rho_t^+}\right|^2\,d\rho_t^+ +\frac{1}{\alpha} \int_{\R^{d-1}}f_t^2(\tilde x)\, d\Hm^{d-1}(\tilde x) \right) dt\right\},
\end{equation}
where we use the identification $\M_+(\R^d_- \sqcup \R^d_+) = \M_+(\R^d_-) \times \M_+( \R^d_+)$. The infimum is taken over all pairs of distributional solutions to the continuity equations
\begin{align}\label{eq: conteqflux}
\partial_t\rho_t^\pm+\divergence V_t^\pm  \pm f_t\Hm^{d-1}|_{\partial \R^d_\pm} = 0\text{ in }\D'((0,1)\times \R^d_\pm),
\end{align}
meaning that in both half-spaces, for all $\phi^\pm\in\C^\infty_c((0,1)\times \R^d_\pm)$ the following equation holds
	\begin{align*}
	&\int_0^1\left(\langle \rho_t^\pm, \partial_t\phi^\pm_t \rangle + \langle V_t^\pm, \nabla \phi_t^\pm \rangle  \mp \langle f_t, [\phi^\pm_t] 
	\rangle\, \right)dt \\
	=&\int_0^1\left(\int_{\R^d_\pm} \partial_t\phi^\pm_t \, d\rho_t^\pm + \int_{\R^d_\pm} \nabla\phi^\pm_t \cdot\, dV^\pm_t \mp \int_{\partial \R^d_\pm}\phi^\pm_t f_t \, d\Hm^{d-1}\right) dt = 0.
	\end{align*}
Here $V^\pm_t\in \M(\R_\pm^d;\R^d)$ is a vector-valued finite Radon measure which is absolutely continuous with respect to $\rho^\pm_t$ and $f_t\in L^2(\R^{d-1}) = L^2(\partial \R^d_\pm)$ is the flux through the membrane, with positive sign denoting flux from the lower into the upper half-space.

\begin{thm}\label{thm:membrane}
Let $\eps>0$.
Then, as $\eps\to 0$, the energies $E_{h^\eps}\colon \CM_+(\R^d)\to[0,\infty]$ $\Gamma$-converge to the limit functional $E_0 \colon \CM_+(\R^d_- \sqcup \R^d_+)\to[0,\infty]$ in the sense that 
\begin{enumerate}
\item[ \emph{(lower bound)}] if $\rho_t^\eps |_{\R^{d}_-} \weakstar \rho_t^-$ in $\M_+(\R^{d}_-)$ for every $t\in [0,1]$, and $\rho_t^\eps|_{(\R^d_+)^\circ} \weakstar \rho_t^+$ in $\M_+(\R^{d}_+)$ for every $t\in [0,1]$ then 
\begin{align}
\liminf_{\eps\to0} E_{h^\eps}((\rho_t^\eps)_{t\in[0,1]})\geq E_0((\rho_t^-,\rho_t^+)_{t\in[0,1]})
\end{align}
\item[\emph{(upper bound)}] for all curves $(\rho_t^-,\rho_t^+)_{t\in[0,1]}$ with $E_0((\rho_t^-,\rho_t^+)_{t\in[0,1]})<\infty$ there exists a sequence $(\rho_t^\eps)_{t\in[0,1]}$ in $\CM_+(\R^d)$ with $\rho_t^\eps |_{\R^{d}_-} \weakstar \rho_t^-$
in $\M_+(\R^{d}_-)$ for every $t\in [0,1]$, and $\rho_t^\eps|_{(\R^d_+)^\circ} \weakstar \rho_t^+$ in $\M_+(\R^{d}_+)$ for every $t\in [0,1]$ and
\begin{align}
\limsup_{\eps\to0}E_{h^\eps}((\rho_t^\eps)_{t\in[0,1]})\leq E_0((\rho_t^-,\rho_t^+)_{t\in[0,1]}).
\end{align}
\end{enumerate}

\end{thm}

We prove this theorem in Section \ref{section: membrane}. Note that the part $\rho_t^\eps|_{(\R^d_+)^\circ}$ includes the mass in the membrane, which is locally bounded by $\alpha \eps^2$.

Since $\Gamma$-convergence implies the convergence of minimizers, the associated minimal energies between two measures $(\rho_0^-,\rho_0^+),(\rho_1^-,\rho_1^+)\in \M_+(\R^d_- \sqcup \R^d_+)$ of equal mass converge as well, as do the minimizing curves themselves.

\subsection{Homogenization of periodic constraints}

The second result concerns the effective limit as $\eps\to 0$ for
\begin{align}
h_\eps(x) \coloneqq h\left(\frac{x}{\eps}\right),
\end{align}
where $h:\R^d \to [0,\infty)$ is $\Z^d$-periodic. This problem is related to the periodic homogenization of elliptic functionals, see \cite{braides1998}. In fact it is a special case of $\cA$-quasiconvex homogenization treated in \cite{braides2000}. In particular it includes perforated domains, where $h(x) = \mathds{1}_U$ for some periodic open set $U\subset \R^d$, modelling the optimal flow of an incompressible fluid through a porous medium (see Figure \ref{fig: porous}), which has received a lot of attention in recent years, see e.g. \cite{otto2001}, \cite{vazquez2007}. 
To the best knowledge of the authors this is a new development in the derivation of porous media equations from inhomogenous materials via optimal transport.

\begin{figure}[h]
\begin{center}
\includegraphics{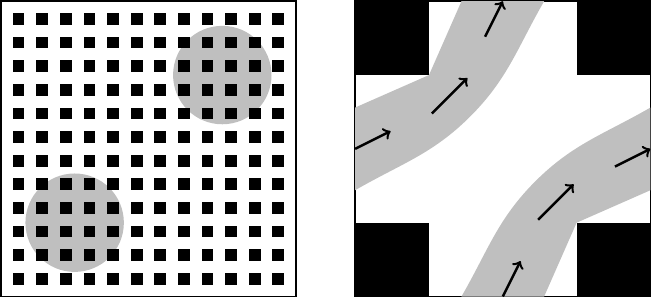}
\end{center}
\caption{Left: Incompressible mass is transported through a region with periodic exclusions. Right: A competitor to the cell problem for $m=\frac14$ and $U=(1,1)$. Note that the exclusion forces a detour increasing the energy.}\label{fig: porous}
\end{figure}

 We will assume throughout the article that $h : \R^d\to [0,\infty)$ satisfies
\begin{enumerate}
	\item [(A1)] $h$ is $\Z^d$-periodic.
	\item [(A2)]$\{h>0\}\subset \R^d$ is open, connected, and Lipschitz bounded.
	\item [(A3)]$h$ is measurable.
	\item [(A4)]$h(x) \in\{0\}\cup [\alpha,\frac1\alpha]$ for some $\alpha\in(0,1]$, for almost every $x\in\R^d$.
\end{enumerate}
We note that $h$ can be interpreted as either a function on the torus $\T^d$ or as a periodic function on $\R^d$.
The connectedness of $\{h>0\}\subset \R^d$ is stronger than connectedness of $\{h>0\} \subset \T^d$.

Under these admissibility assumptions, we show $\Gamma$-convergence of $E_{h_\eps}$ to  the homogenized transport cost $E_{\hom}: \CM_+(\T^d) \to [0,\infty]$,
\begin{align}\label{eq: homogenized}
E_{\hom}((\rho_t)_{t\in[0,1]}) \coloneqq \inf\left\{\int_0^1\int_{\T^d} f_{\hom}(\frac{d\rho_t}{dx},\frac{dV_t}{dx})\, dx\, dt: \partial_t\rho_t+\divergence V_t=0 \text{ in }\mathcal D'((0,1)\times \T^d)\right\},
\end{align}
where only absolutely continuous curves $\rho_t, V_t \ll \Lm^d$ are allowed. Otherwise $E_{\hom}$ is defined to be $\infty$.

The homogenized energy density $f_{\hom}: [0,\infty) \times \R^d \to [0,\infty]$ is given by
\begin{align}\label{eq: f_hom}
f_{\hom}(m,U)\coloneqq\inf\left\{\int_{\T^d}\frac{|W(x)|^2}{\nu(x)}\, dx\right\},
\end{align}
and the infimum is taken among all $\nu\in L^\infty(\T^d)$ such that $0\leq \nu(x) \leq h(x)$ almost everywhere and $\int_{\T^d} \nu(x)\,dx = m$, and all $W\in L^2(\T^d;\R^d)$ such that
$\divergence W=0\text{ in }\mathcal D'(\T^d)$, and $\int_{\T^d}W(x)\,dx = U$.

\begin{thm}\label{thm:homogenizationintro}
Let $h\colon \T^d\to[0,\infty)$ satisfy the assumptions \emph{(A1) - (A4)}.
 Then, as $\eps\to0$, and $\rho_t^\eps\weakstar\rho_t$ in $\M_+(\T^d)$ for every $t\in[0,1]$, $E_{h_\eps} : \CM_+(\T^d) \to [0,\infty]$
$\Gamma$-converges to $E_{\hom}:\CM_+(\T^d) \to [0,\infty]$ in the sense that
\begin{itemize}
\item[\emph{(lower bound)}] if $\rho_t^\eps \weakstar \rho_t$ in $\M_+(\T^d)$ for every $t\in [0,1]$, then 
\begin{align}
\liminf_{\eps\to0} E_{h_\eps}((\rho_t^\eps)_{t\in[0,1]})\geq E_{\hom}((\rho_t)_{t\in[0,1]})
\end{align}
\item[\emph{(upper bound)}] for all curves $(\rho_t)_{t\in[0,1]}$ with $E_{\hom}((\rho_t)_{t\in[0,1]})<\infty$ there exists a sequence $(\rho_t^\eps)_{t\in[0,1]}$ in $\CM_+(\T^d)$ with $\rho_t^\eps \weakstar \rho_t$
in $\M_+(\T^{d})$ for every $t\in [0,1]$, and
\begin{align}
\limsup_{\eps\to 0} E_{h_\eps}((\rho_t^\eps)_{t\in[0,1]})\leq E_{\hom}((\rho_t)_{t\in[0,1]}).
\end{align}
\end{itemize}	
	
\end{thm}

\begin{rem}\label{rem: convexity} In the one-dimensional case $E_{\hom}$ is given by 
 \begin{align}
E_{\hom}((\rho_t)_{t\in[0,1]})=\inf\left\{\int_0^1\int_\R \frac{\left(\frac{dV_t}{dx}\right)^2}{ F(\frac{d\rho_t}{dx})}\, dx\, dt: \partial_t\rho_t+\partial_x V_t=0\right\},
\end{align}
where $F(m) \coloneqq \left(\inf\{\int_0^1\frac1{\nu(x)}\, dx: \nu\leq h,\int_0^1\nu=m\} \right)^{-1}$ is the mobility. Since $f_{\hom}(m,U) = \frac{U^2}{F(m)}$ is convex (see Lemma \ref{lma: prophom}), this means that the mobility $m \mapsto F(m)\leq m$ must be concave, which signifies a congestion effect.
\end{rem}

In Section \ref{sec:compactness} we prove lower-semicontinuity and compactness of the functionals $E_h$, $E_0$ and $E_{\hom}$, which is relevant for later sections. In Section \ref{sec:duality} we find the dual problems of \eqref{eq: problem}, \eqref{eq: membrane} and \eqref{eq: homogenized} and characterize the minimizers by the respective Euler-Lagrange equations. In Section \ref{sec:gradient flows} we state the PDE solved by the steepest descent of the Helmholtz free energy functional
$\rho\mapsto RT\int_{\Omega}\rho(x)\log \rho(x)\, dx+\int_{\Omega}\rho(x)\psi(x)\, dx$
for each cost. Additionally, in Section \ref{sec:stark} we give an example of an optimal curve under a nontrivial density constraint.
Finally, in Section \ref{section: membrane} and Section \ref{section: homogenization} we prove Theorem \ref{thm:membrane} and Theorem \ref{thm:homogenizationintro} respectively.

\section{Compactness}\label{sec:compactness}

\begin{lma}\label{lma: lower semicontinuity}
The functionals $E_h$, $E_0$, and $E_{\hom}$ defined in \eqref{eq: problem}, \eqref{eq: membrane}, and \eqref{eq: homogenized} are lower semicontinuous with respect to pointwise weak-$\ast$ convergence on $\CM_+(\Omega)$, where $\Omega$ is $\R^d$ or $\T^d$, $R^d_-\sqcup\R^d_+$, or $\T^d$ respectively.

Given a fixed finite Radon measure $\rho_0\in \M_+(\Omega)$, all families of curves starting in $\rho_0$ with bounded energies have a subsequence converging pointwise weak-$\ast$ in $\CM_+(\Omega)$.

\end{lma}

\begin{proof}
Compactness in the constrained and homogenized case follows from the fact that in \eqref{eq: problem}, \eqref{eq: homogenized}, the functional is bounded from below by the Wasserstein action \eqref{eq: benamou}. For \eqref{eq: homogenized}, this bound is shown in Lemma \ref{lma: prophom}. The compactness then follows from the tightness of balls in Wasserstein space and the uniform continuity of sequences of curves with finite energy.

We show compactness in the membrane case in two steps. First, for $0<r<R$ define a test function $\eta_{r,R}\in \mathcal C_c^\infty(\R^d_- \sqcup \R^d_+)$ such that $\eta_{r,R}(x) = \eta_{r,R}(|x|)$, $\eta_{r,R} = 0$ outside of $B(0,2R)\setminus B(0,r/2)$, $\eta = 1$ in $B(0,R) \setminus B(0,r)$, and $|\nabla \eta_{r,R}| \leq \frac{C}{r}$. Then
\begin{equation}
\begin{aligned}
\langle \rho_t^n, \eta_{r,R} \rangle = &\langle \rho_0, \eta_{r,R} \rangle + \int_0^t \langle \partial_s \rho_s^n, \eta_{r,R} \rangle \,ds = \langle \rho_0, \eta_{r,R} \rangle + \int_0^t \langle V_s^n, \nabla \eta_{r,R} \rangle \,ds\\
\leq &\langle \rho_0, \eta_{r,R} \rangle + \frac{C}{r} \left(E_0((\rho_s^n)_{s\in [0,1]})\right)^{1/2} \left(\int_0^1\rho_s^n(\R^d_- \sqcup \R^d_+)\,ds\right)^{1/2}.
\end{aligned}
\end{equation}
The last term is uniformly small in $n$ as $r\to \infty$, showing tightness of the $(\rho_t^{-,n},\rho_t^{+,n})_{t\in[0,1],n\in\N} \subset \M_+(\R^d_- \sqcup \R^d_+)$.

Now pick a countable family $(\eta_i)_{i\in\N} \subset \mathcal C_c^\infty(\R^d_- \sqcup \R^d_+)$ that is dense in $\mathcal C_c^0(\R^d_- \sqcup \R^d_+)$. Then whenever $0\leq t_0 \leq t_1 \leq 1$, $n,i\in\N$, we have 
\begin{equation}
\begin{aligned}
\left| \langle \rho_{t_0}^n - \rho_{t_1}^n, \eta_i \rangle \right| = &\left|\int_{t_0}^{t_1} \langle \partial_t \rho_t^n, \eta_i \rangle \,dt\right|\\
 = &\left|\int_{t_0}^{t_1} \langle V_t^n, \eta_i \rangle - \langle f_i, [\eta_i] \rangle\,dt\right|\\
 \leq & C(\eta_i) \left(t_1 - t_0\right)^{1/2} \left(E_0((\rho_t^{-,n},\rho_t^{+,n})_{t\in [0,1]})\right)^{1/2}.
\end{aligned}
\end{equation}
It follows that $\lim_{h \to 0} \sup_{n,t} \left|\langle \rho_{t+h}^n- \rho_t^n, \eta_i \rangle\right| = 0 $ for every $i$. By Helly's Selection Theorem there exists a subsequence $(\rho_t^{n_k})_{t,\in [0,1], k\in \N}$ and a curve $(\rho_t^-,\rho_t^+)_{t\in[0,1]} \subset \M_+(\R^d_- \sqcup \R^d_+)$ such that $\langle \rho_t^{n_k}, \eta_i \rangle \to \langle \rho_t, \eta_i \rangle$ for every $i\in N$ and every $t\in[0,1]$. By tightness, $\rho_t^{n_k} \weakstar \rho_t$ for every $t\in [0,1]$, which proves the compactness.

In cases \eqref{eq: problem} and \eqref{eq: homogenized}, to prove the lower bound, take a sequence of curves $(\rho_t^n,V_t^n)_{t\in[0,1],n\in \N}$ with finite energy. We see by H\"older's inequality that
\begin{equation}
\int_0^1|V_t^n|(B(x,R))\,dt = \int_0^1 \int_{B(x,R)} \left|\frac{dV_t^n}{d\rho_t^n}\right|\,d\rho_t^n\,dt \leq \left(E((\rho^n_t)_{t\in[0,1]})\right)^{1/2}\rho_t^n(B(x,R))^{1/2},
\end{equation}
where $E=E_{h}$ or $E=E_{\hom}$, respectively.
The right hand side is bounded, so that a subsequence $V_t^n$ converges vaguely (not necessarily weak-$\ast$ in the case of \eqref{eq: membrane}) to some $V\in\M([0,1]\times \Omega;\R^d)$. We note that $V_t$ is absolutely continuous with respect to $dt$, so that by the disintegration theorem $V = \int_0^1 V_t\,dt$ for some $V_t\in\M(\Omega;\R^d)$ defined for almost every $t$, and $\partial_t \rho_t + \divergence V_t = 0$ in $\D'((0,1)\times \Omega)$.

The same argument works in case $\eqref{eq: membrane}$, yielding finite measures $(V_t^-,V_t^+)_{t\in[0,1]} \subset \M(\R^d_- \sqcup \R^d_+;\R^d)$. Additionally, $f_t^n \rightharpoonup f_t$ in $L^2([0,1] \times \R^{d-1})$. The limits then solve $\partial_t \rho_t^\pm + \divergence V_t^\pm  \pm f_t\Hm^{d-1}|_{\partial \R^d_\pm}=0$ in $\D'((0,1)\times (\R^d_- \sqcup \R^d_+))$, and by Fubini's theorem and the weak lower semicontinuity of the norm
\begin{equation}
\int_0^1 \int_{\R^{d-1}} \frac{1}{\alpha} f_t^2\,d\Hm^{d-1}\,dt \leq \liminf_{n\to\infty} \int_0^1 \int_{\R^{d-1}} \frac{1}{\alpha} (f_t^n)^2\,d\Hm^{d-1}\,dt.
\end{equation}

To show lower semicontinuity of the remaining term
\begin{equation}
\int_0^1\int_{\Omega} \left|\frac{dV_t}{d\rho_t}\right|^2\,d\rho_t\,dt
\end{equation}
with $\Omega \in \{\R^d, \R^d_- \sqcup \R^d_+, \T^d\}$, we use Theorem 2.34 in \cite{AmFuPa}, which states that for $g:\R^M \to [0,\infty]$ convex, lower semicontinuous, with recession function $g^\infty:\R^M \to [0,\infty]$, the functional defined on $\M(\Omega;\R^M)$
\begin{equation}
G(P) \coloneqq \int_\Omega g\left(\frac{dP}{dx}\right)\,dx + \int_\Omega g^\infty\left(\frac{dP}{d|P|}\right)\,d|P|^s
\end{equation}
is vaguely sequentially lower semicontinuous. We apply this to the sequence $P^n_t \coloneqq (\rho_t^n,V_t^n) \in \M(\Omega;\R^{d+1})$, which in any case converges vaguely to $P_t \coloneqq (\rho_t,V_t)$. The function $g$ is either $g(m,U) = \frac{|U|^2}{m}$ in cases \eqref{eq: problem} and \eqref{eq: membrane} or $g=f_{\hom}$ in case \eqref{eq: homogenized}. We now have to do some extra work depending on the case:

In case \eqref{eq: problem}, we have to show that $\rho_t(A) \leq \int_A h(x)\,dx$ for every open set $A\subset \Omega$. This is the Portmanteau theorem, found in e.g. \cite[Example 1.63]{AmFuPa}.

In case \eqref{eq: homogenized}, we know from Lemma \ref{lma: prophom} that $f_{\hom}$ is convex and lower semicontinuous. We have to make sure that the singular part of $(\rho_t,V_t)$ vanishes. Indeed, this holds if $h \in L^1(\T^d)$, as in that case $\frac{d\rho_t^n}{dx} \leq \int_{\T^d} h(y)\,dy$ for every $n$, and this property is inherited by $\rho_t$. Moreover, $V_t \ll \rho_t$ if the energy is finite.

In case \eqref{eq: membrane}, we have nothing more to show. This completes the proof.
\end{proof}

\section{Duality and minimality}\label{sec:duality}

In this section, we characterize the dual problems to \eqref{eq: problem}, \eqref{eq: membrane}, and \eqref{eq: homogenized} and find the Euler-Lagrange equations. To this end, we fix endpoints $\rho_0,\rho_1\in \M_+(\Omega)$ with finite and equal mass.

\subsection{Constrained optimal transport}

Here we minimize the action functional
\begin{equation}
F((\rho_t, V_t)_{t\in[0,1]}) \coloneqq \int_0^1 \int_\Omega \frac{1}{2}\left|\frac{dV_t}{d\rho_t}\right|^2\,d\rho_t \,dt
\end{equation}
subject to $0\leq \rho_t(A) \leq \int_A h(x)\,dx$ for every open $A\subset \Omega$ and $\partial_t \rho_t + \divergence V_t = 0$ in $\D'((0,1)\times \Omega)$, and $\rho_0,\rho_1$ fixed. We introduce the Lagrange multiplier $\phi_t\in \C^1([0,1]\times \Omega)$ and write by Sion's minimax theorem
\begin{equation}
\begin{aligned}
&\inf_{0\leq \rho_t \leq h\,dx, V_t}\sup_{\phi_t} F((\rho_t, V_t)_{t\in [0,1]}) + \int_0^1 \langle \partial_t \rho_t + \divergence V_t, \phi_t \rangle\,dt\\
= & \sup_{\phi_t} \langle \phi_1, \rho_1 \rangle - \langle \phi_0, \rho_0 \rangle - \inf_{0\leq \rho_t \leq h\,dx} \int_0^1 \langle \partial_t \phi_t+ \frac{1}{2}|\nabla \phi_t|^2, \rho_t \rangle \,dt\\
= &\sup_{\phi_t} \langle \phi_1, \rho_1 \rangle - \langle \phi_0, \rho_0 \rangle - \int_0^1 \langle (\partial_t \phi_t + \frac{1}{2}|\nabla \phi_t|^2)_+, h \rangle \,dt.
\end{aligned}
\end{equation}
The last term is the constrained dual problem. Note that wherever $h=\infty$, we formally recover the Kantorovich dual.

By the complementary slackness theorem, the minimizer $(\rho_t)_{t\in[0,1]}$ and maximizer $(\phi_t)_{t\in[0,1]}$ are characterized by the Euler-Lagrange equations
\begin{equation}
\begin{cases}\label{eq: geodesic flow}
\partial_t \rho_t + \divergence ( \rho_t\nabla \phi_t) = 0\\
\partial_t \phi_t + \frac{1}{2} |\nabla \phi_t|^2 = p_t\\
\rho_t p_t \geq 0\\
\rho_t(h-\rho_t) p_t = 0.
\end{cases}
\end{equation}
Here $p_t:\Omega \to [0,\infty)$ is the Lagrange multiplier to the constraint on $\rho_t$ acting as a pressure on the potential.

\subsection{Optimal membrane transport}

Here we minimize
\begin{equation}
F((\rho_t^\pm, V_t^\pm,f_t)_{t\in[0,1]}) \coloneqq  \int_0^1 \left(\sum_\pm\int_{\R^d_\pm} \frac{1}{2}\left|\frac{dV_t^\pm}{d\rho_t^\pm}\right|^2\,d\rho^\pm_t + \int_{\R^{d-1}} \frac{1}{2\alpha}|f_t|^2 \right) \,dt
\end{equation}
subject to $0\leq \rho_t^\pm$ and $\partial_t \rho_t^\pm + \divergence V_t^\pm  \pm f_t\Hm^{d-1}|_{\partial \R^d_\pm}=0$ in $\D'((0,1)\times (\R^d_- \sqcup \R^d_+))$, and $\rho_0^\pm,\rho_1^\pm$ fixed. We introduce the Lagrange multipliers $\phi_t^\pm \in \C^1([0,1]\times \R^d_\pm)$ and write by Sion's minimax theorem, denoting $[\phi_t](\tilde x) = \phi_t^+(\tilde x) - \phi_t^-(\tilde x):\R^{d-1}\to \R$,
\begin{equation}
\begin{aligned}
&\inf_{0\leq \rho_t^\pm, V_t^\pm,f_t}\sup_{\phi_t^\pm} F((\rho_t^\pm, V_t^\pm,f_t)_{t\in [0,1]}) + \int_0^1 \langle \partial_t \rho_t + \divergence V_t, \phi_t \rangle + \langle [\phi_t],f_t \rangle \,dt\\
= & \sup_{\phi_t^\pm} \sum_\pm \left(\langle \phi_1^\pm, \rho_1^\pm \rangle - \langle \phi_0^\pm, \rho_0^\pm \rangle - \inf_{0\leq \rho_t^\pm} \int_0^1 \langle \partial_t \phi_t^\pm+ \frac{1}{2}|\nabla \phi_t^\pm|^2, \rho_t^\pm \rangle \,dt\right) - \int_0^1\int_{\R^{d-1}}\frac{\alpha}{2} [\phi_t]^2\,d\Hm^{d-1}\,dt\\
= &\sup_{\partial_t\phi_t^\pm + \frac{1}{2}|\nabla \phi_t^\pm|^2 \leq 0} \sum_\pm \left( \langle \phi_1^\pm, \rho_1^\pm \rangle - \langle \phi_0^\pm, \rho_0^\pm \rangle \right) - \int_0^1 \int_{\R^{d-1}} \frac{\alpha}{2} [\phi_t]^2\,d\Hm^{d-1}\, dt.
\end{aligned}
\end{equation}
The last term is the constrained dual problem. Note that as $\alpha \to \infty$, we formally recover the Kantorovich dual problem in $\R^d$, whereas as $\alpha \to 0$, we formally recover two separate Kantorovich dual problems in $\R^d_\pm$.

By the complementary slackness theorem, the minimizers $(\rho_t^\pm)_{t\in[0,1]}$ and maximizers $(\phi_t^\pm)_{t\in[0,1]}$ are characterized by the Euler-Lagrange equations
\begin{equation}
\begin{cases}
\partial_t \rho_t^\pm + \divergence (\rho_t^\pm\nabla \phi_t^\pm ) \mp \alpha[\phi_t]\Hm^{d-1}|_{\partial \R^d_\pm} = 0\\
\partial_t \phi_t^\pm + \frac{1}{2} |\nabla \phi_t|^2 = p_t^\pm\\
\rho_t^\pm p_t^\pm = 0\\
 p_t^\pm \leq 0.
\end{cases}
\end{equation}

\subsection{Homogenized optimal transport}

Here we minimize
\begin{equation}
F((\rho_t, V_t)_{t\in[0,1]}) \coloneqq \int_0^1  \int_{\T^d} f_{\hom}(\rho_t, V_t)\,dx \,dt
\end{equation}
subject to $\partial_t \rho_t + \divergence V_t = 0$ in $\D'((0,1)\times \T^d)$, and $\rho_0,\rho_1$ fixed. We introduce the Lagrange multiplier $\phi_t\in \C^1([0,1]\times \T^d)$ and write by Sion's minimax theorem
\begin{equation}
\begin{aligned}
&\inf_{\rho_t, V_t}\sup_{\phi_t} \int_0^1 \int_{\T^d} f_{\hom}(\rho_t,V_t) +  (\partial_t \rho_t + \divergence V_t) \phi_t \,dx\,dt\\
= & \sup_{\phi_t} \langle \phi_1, \rho_1 \rangle - \langle \phi_0, \rho_0 \rangle - \int_0^1 \int_{\T^d} f_{\hom}^\ast(\partial_t \phi_t, \nabla \phi_t)\,dx \,dt.
\end{aligned}
\end{equation}
The last term is the dual problem. We check that for $f_{\hom}(m,U) = \frac{|U|^2}{2m}$ on $[0,\infty) \times \R^d$, we have
\begin{equation}
f_{\hom}^\ast(\partial_t \phi_t, \nabla \phi_t) =
\begin{cases}
0,&\text{ if }\partial_t \phi_t + \frac{1}{2}|\nabla \phi_t|^2 \leq 0\\
\infty,&\text{ otherwise,}
\end{cases}
\end{equation}
as expected.

By the complementary slackness theorem, the minimizer $(\rho_t)_{t\in[0,1]}$ and maximizer $(\phi_t)_{t\in[0,1]}$ are characterized by the Euler-Lagrange differential inclusions, which are stated in terms of the partial Legendre transform $f_{\hom}^{\ast U}(m,P) \coloneqq \sup_U P\cdot U - f_{\hom}(m,U)$ as
\begin{equation}
\begin{cases}
V_t \in \partial_P^- f_{\hom}^{\ast U}(\rho_t, \nabla \phi_t)\\
\partial_t \rho_t + \divergence V_t =  0\\
\partial_t \phi_t \in \partial_m^- f_{\hom}(\rho_t, V_t).
\end{cases}
\end{equation}
This is a general formulation of congested mean field games. A similar model of congested mean field game is treated in e.g. \cite{benamou2017variational}.

To be more specific, in the idealized case $f_{\hom}(m,U) = \frac{|U|^2}{2m^{1-\beta}}$, $\beta\in [0,1)$, the mean field game equation is given by
\begin{align}
\partial_t \rho_t + \divergence \rho_t^{1-\beta} \nabla \phi_t = 0,\quad
\partial_t \phi_t + \frac{1-\beta}{2}\frac{|\nabla \phi_t|^2}{\rho_t^\beta} = 0.
\end{align} 

\section{Gradient flows}\label{sec:gradient flows}

We now look at the formal constrained gradient flows of the functionals
\begin{equation}
\rho\mapsto RT \int_{\Omega}\rho(x)\log\rho(x)\,dx + \int_\Omega \rho(x) \psi(x)\, dx
\end{equation}
with $\psi\in \mathcal C^1(\Omega)$ the Gibbs free energy, and $R,T>0$ the gas constant and the temperature respectively.

We will write down the PDE corresponding to steepest descent of $F$ with costs given by \eqref{eq: problem}, \eqref{eq: membrane}, and \eqref{eq: homogenized}. Without loss generality we assume $RT=1$.

\subsection{Constrained gradient flow}

Given $\rho\in L^1$, we want to find $V\in L^1_{\loc}(\Omega;\R^d)$ minimizing
\begin{equation}
\begin{aligned}
&\int_\Omega - ((\log\rho(x) + 1) + \psi(x)) \divergence  V(x) + \frac{|V(x)|^2}{2\rho(x)}\,dx \\
= &\int_\Omega ( \nabla \log\rho(x) + \nabla \psi(x)) \cdot V(x) + \frac{|V(x)|^2}{2\rho(x)}\,dx
\end{aligned}
\end{equation}
subject to $\divergence V \geq 0$ on $\{\rho = h\}$. We introduce a Lagrange multiplier $ p \in L^1(\Omega)$, $p \geq 0$, $p(h-\rho) = 0$, and write the above problem as
\begin{equation}
\begin{aligned}
 &\min_V \sup_{p \geq 0, p(h-\rho) =0} \int_\Omega ( \nabla \log\rho(x) + \nabla \psi(x)) \cdot V(x) + \frac{|V(x)|^2}{2\rho(x)} - p(x)\divergence V(x)\,dx\\
 = &  \sup_{p \geq 0, p(h-\rho) =0} \min_V \int_\Omega ( \nabla \log\rho(x) + \nabla \psi(x) + \nabla p(x)) \cdot V(x) + \frac{|V(x)|^2}{2\rho(x)}\,dx\\
 = &  \sup_{p \geq 0, p(h-\rho) =0} \int_\Omega -\frac{\rho(x)}{2}| \nabla \log\rho(x) + \nabla \psi(x) + \nabla p(x)|^2\,dx\\
\end{aligned}
\end{equation}

We see that the minimizer can be written $V(x) = -\rho\nabla \phi(x)$, where $\phi:\Omega \to \R$ solves the elliptic obstacle problem
\begin{equation}
 \begin{cases}
  \phi(x) \geq \log \rho(x) + \psi(x)\\
  \phi(x) =  \log \rho(x) + \psi(x) \text{ in }\{\rho < h\}\\
  \phi\text{ maximizes }\int_\Omega -\frac{\rho(x)}{2}|\nabla \phi(x)|^2\, dx.
 \end{cases}
\end{equation}

Physically, the difference between $\phi(x)$ and the chemical potential $ \log\rho(x) + \psi(x)$ acts as a hydrostatic pressure $p(x)\geq 0$ with $p(x)(h(x)-\rho(x)) = 0$.

Inserting $V$ into the continuity equation yields a constrained version of the Fokker-Planck equation, 
\begin{equation}\label{eq: constrained heat}
\partial_t \rho_t - \divergence (\rho_t \nabla \phi_t) = 0.
\end{equation}
 Note that in \cite{jordan1998}, the authors rigorously derive the unconstrained Fokker-Planck equation as the $W_2$-gradient flow of $F$. We note that this version of the constrained Fokker-Planck equation differs from the Stefan problem treated in e.g. \cite{moerbeke1976}, which is not mass-preserving.

\subsection{Membrane gradient flow}
Here, given $(\rho^-,\rho^+)\in L^1(\R^d_- \sqcup \R^d_+)$, and a Gibbs free energy $(\psi^-,\psi^+)\in \mathcal C^1(\R^d_- \sqcup \R^d_+)$, we find $V^-,V^+,f$ minimizing
\begin{equation}
\sum_\pm \int_{\R^d_\pm}  ( \nabla \log \rho^\pm(x) + \nabla \psi^\pm(x)) \cdot V^\pm(x) + \frac{|V^\pm(x)|^2}{2\rho^\pm(x)}\,dx + \int_{\R^{d-1}} -[\log\rho + \psi](\tilde x) f(\tilde x) + \frac{f(\tilde x)^2}{2\alpha} \,d\tilde x.
\end{equation}
Inserting the minimizers into the continuity equation yields two Fokker-Planck equations coupled through the Teorell equation on the membrane \cite{teorell},
\begin{equation}
\begin{cases}
\partial_t \rho_t^\pm - \Delta \rho_t^\pm - \divergence(\rho_t^\pm \nabla \psi^\pm) = 0, &\text{ in }\R^d_\pm\\
\nabla \rho_t^\pm \cdot e_d = \alpha [\log\rho_t + \psi],&\text{ on }\partial \R^d_\pm.
\end{cases}
\end{equation}

\subsection{Homogenized gradient flow}
Given $\rho\in L^1(\T^d)$, we find $V\in L^1_{\loc}(\T^d;\R^d)$ minimizing
\begin{equation}
\int_{\T^d} \nabla (\log \rho(x) + \psi(x)) \cdot V(x) + f_{\hom}(\rho(x),V(x))\,dx.
\end{equation}
We see that $V(x) \in \partial_P^- f_{\hom}^{\ast U}(\rho, -  \nabla \log \rho(x) - \nabla \psi(x))$. Note that for $f_{\hom}(m,U) = \frac{|U|^2}{2m^{1-\beta}}$, $\beta \in (0,1)$, which is a reasonable choice according to Remark \ref{rem: convexity}, and $\psi = 0$, we recover the porous medium equation
\begin{equation}\label{eq: porous}
0 = \partial_t \rho_t + \divergence (- \rho_t^{1-\beta} \nabla \log \rho_t) = \partial_t \rho_t - \frac{1}{1-\beta} \Delta \rho_t^{1-\beta}.
\end{equation}
We note that Theorem \ref{thm:homogenizationintro} does not imply convergence of gradient flows \eqref{eq: constrained heat} with $h=h_\eps$ to \eqref{eq: porous}.

\section{The stark constraint}\label{sec:stark}
In the following we give a simple one-dimensional example of an optimal curve under a nontrivial density constraint, which bounds the density by $\lambda>0$ on $(0,\infty)$.
\begin{equation}
h(x) = \begin{cases}
       \lambda, &\text{ if }x\in (0,\infty)\\
       \infty, &\text{ otherwise},
       \end{cases}
\end{equation}
where $\lambda,m>0$.  We call this the stark constraint.
We construct an optimal curve $(\rho_t)_{t\in[0,1]}$ starting in $\rho_0 =m \delta_0$ and ending in the uniform density $\rho_1 = \lambda \mathds{1}_{(0,\frac{m}\lambda)}\,dx$.

We choose this example because the solution breaks conservation of momentum, while kinetic energy is conserved. The calculations in this case are straightforward but already quite lengthy. The complexity only increases in higher dimensions and with more variation in $h$.

We choose the following ansatz for the optimal curve:
\begin{align}
\rho_t =\lambda\Lm|_{(0,x_t)} + (m - \lambda x_t)\delta_0.
\end{align}
 We see that any $\rho_t(A) \leq \lambda \Lm(A)$ for any Borel $A\subseteq (0,\infty)$, and the boundary conditions are satisifed if and only if $x_0 = 0$ and $x_1 = \frac m\lambda$. The momentum field $V_t=\lambda \dot x_t\Lm|_{(0,x_t)}$ solves the continuity equation 
\begin{align}
\partial_t\rho_t+\partial_xV_t=0
\end{align}
with action given by
\begin{align}
\int_0^1\int_{\R} \left(\frac{dV_t}{d\rho_t}\right)^2\, d\rho_t\, dt=\int_0^1\lambda\dot x_t^2 x_t\, dt=\int_0^1|\frac{d}{dt}G(x_t)|^2\, dt,
\end{align}
where $G(y)=\frac23 \lambda^{1/2}y^{3/2}$. The minimizer satisfies $\frac{d}{dt}G(x_t)=c$, where $c$ is the unique constant compatible with the boundary conditions $x_0=0$ and $x_1=\frac{m}{\lambda}$. 
We see that 
\begin{align}
x_t=
\frac m\lambda t^{\frac23},
\end{align}
and consequently
\begin{align}
\int_0^1|\frac{d}{dt}G(x_t)|^2\, dt=(G(x_1)-G(x_0))^2=\frac49\frac{m^3}{\lambda^2}.
\end{align}

We claim that $\rho_t$ is optimal among all curves independent of the ansatz. To see this we consider the dual problem. Let 
\begin{align}
\phi_t(x)=\frac23\frac m\lambda  t^{-\frac13}x_+,
\end{align}
where $\phi_0(0)=0$ and $\phi_0(x)=\infty$ for $x>0$.
Formally, we have
\begin{equation}
\begin{aligned}
& \lim_{\eps \to 0}\langle \phi_1,\rho_1\rangle-\langle \phi_\eps,\rho_\eps\rangle-\int_\eps^1\int_{\R} (\partial_t\phi_t+\frac12|\partial_x\phi_t|^2)_+ h\, dx\, dt\\
= &\frac13 \frac{m^3}{\lambda^2} - 0 - \int_0^1 \int_0^{x_t} \frac29 m t^{-2/3}\left( - t^{-2/3} x + \frac m\lambda \right)\,dx\, dt\\
= & \frac 29\frac{m^3}{\lambda^2},
\end{aligned}
\end{equation}
which is half the primal cost $\int_0^1|\frac{d}{dt}G(x_t)|^2\, dt$. By duality $\rho_t$ and $\phi_t$ must be optimal. In fact, they formally solve \eqref{eq: geodesic flow} with pressure
\begin{equation}
p_t(x) = \frac 29 \frac m\lambda t^{-2/3}\left( \frac m \lambda - t^{-2/3} x \right) \mathds{1}_{(0,x_t)}(x).
\end{equation}

 However, at $x=0$, $\phi_t$ is not differentiable and at $t=0$, it is not continuous. 
To make the optimality precise, we approximate $\phi$ with $\mathcal C^1$-functions $\phi_t^\eps(x) = \frac23 \frac{m}{\lambda}t^{-\frac13}\eta^\eps(x)$, with $\eta^\eps\to x_+$ uniformly and $(\eta^\eps)' - \1_{[0,\infty)} \to 0$ in $L^1(\R)$. Then for every $\delta >0$, we have
\begin{align}
\lim_{\eps \to 0} \langle \phi_{1-\delta}^\eps,\rho_{1-\delta}\rangle-\langle \phi_\delta^\eps,\rho_\delta\rangle-\int_\delta^{1-\delta}\int_{\R} (\partial_t\phi_t^\eps+\frac12|\partial_x\phi_t^\eps|^2)_+ h\, dx\, dt=\frac12 \int_\delta^{1-\delta}\left|\frac{d}{dt}G(x_t)\right|^2\,dt.
\end{align}
This shows that $(\rho_t)_{t\in[\delta,1-\delta]}$ is optimal. Letting $\delta \to 0$, optimality of $(\rho_t)_{t\in [0,1]}$ follows.

\section{The membrane limit}\label{section: membrane}
We now prove Theorem \ref{thm:membrane}. Recall that $h^\eps:\R^d\to[0,\infty]$ is given by the stark constraint
\begin{equation}
h^\eps(x) =
\begin{cases}
\alpha \eps, &\text{ if } x_d\in (0,\eps)\\
\infty, &\text{ elsewhere,}
\end{cases}
\end{equation}
with $\alpha\in (0,\infty)$ fixed and $\eps \to 0$. 

Because $\Gamma$-convergence is compatible with partial minimization, the minimum costs for all curves also $\Gamma$-converge. 

\begin{proof}[Proof of the lower bound]
Consider a family of curves of bounded nonnegative measures $(\rho_t^\eps)_{t\in [0,1],\eps>0} \subset \M_+(\R^d)$ with $\rho_t^\eps(dx) \leq h^\eps(x)\, dx$, where $\rho_t^\eps|_{\R^{d-1}\times (-\infty,0]}\weakstar \rho_t^-$ and $\rho_t^\eps|_{\R^{d-1}\times (0,\infty)}\weakstar \rho_t^+$. Also find the respective minimizing momentum fields $(V_t^\eps)_{t\in [0,1],\eps>0} \subset \M(\R^d, \R^d)$ such that $\partial_t \rho_t^\eps + \divergence V_t^\eps = 0$ in $\D'((0,1)\times \R^d)$ and
\begin{equation}\label{eq: lower bound energy}
E_{h^\eps}((\rho_t^\eps)_{t\in [0,1]}) = \int_0^1 \int_{\R^d} \left|\frac{dV_t^\eps}{d\rho_t^\eps}\right|^2\,d\rho_t^\eps\,dt.
\end{equation}

We shall assume throughout the proof that \eqref{eq: lower bound energy} is bounded by some constant independent of $\eps$ by extracting a subsequence, as without the existence of a bounded energy subsequence there is nothing to prove.

We now employ the standard dimension reduction technique of blowing up the thin constrained region, as was done in e.g. \cite{FJM}. We introduce the notation $x = (\tilde x, x_d) \in \R^d$. To that end, let $T_\eps: \R^d \to \R^d$ be defined by
\begin{equation}
T_\eps(x) = \begin{cases}
            x - (1-\eps)e_d, &\text{ if }x_d\geq 1\\
            (\tilde x, \eps x_d), &\text{ if }x_d\in (0,1)\\
            x, &\text{ if }x_d \leq 0,
            \end{cases}
\end{equation} 
so that $T_\eps(\R^{d-1}\times (0,1)) = \R^{d-1}\times (0,\eps)$.

We define $\pi_t^\eps = (T_\eps)_\# \rho_t^\eps$ and $W_t^\eps(x) = DT_\eps(T_\eps^-1(x)) V_t^\eps(T_\eps^{-1}(x))$, i.e.
\begin{equation}
W_t^\eps(x) = \begin{cases}
              V_t^\eps(x - (1-\eps)e_d),&\text{ if }x_d \geq 1\\
              (\eps\tilde V_t^\eps(\tilde x, \eps x_d), (V_t^\eps)_d(\tilde x, \eps x_d), &\text{ if }x_d\in (0,1)\\
              V_t^\eps(x), &\text{ if }x_d \leq 0. 
              \end{cases}
\end{equation}
By this choice, $\partial_t \pi_t^\eps + \divergence W_t^\eps = 0$, and
\begin{equation}
\int_0^1 \int_{\R^{d-1} \times (0,\eps)} \frac{|V_t^\eps|^2}{\rho_t^\eps} \,dx \,dt = \int_0^1 \int_{\R^{d-1} \times (0,1)} \frac{|\tilde W_t^\eps|^2 + \eps^2 (W_t^\eps)_d^2}{\pi_t^\eps}\,dx \,dt.
\end{equation}

Because $\pi_t^\eps \leq \alpha \eps^2$ in $\R^{d-1} \times (0,1)$, it follows that $\tilde W_t^\eps \to 0$ strongly in $L^2([0,1] \times \R^{d-1}\times (0,1))$, and that a subsequence of $(W_t^\eps)_d$ converges weakly in $L^2([0,1] \times \R^{d-1} \times (0,1))$ to some $f_t \in L^2([0,1] \times \R^{d-1} \times (0,1))$. In addition, $\pi_t^\eps \to 0$ in $L^\infty([0,1]\times \R^{d-1} \times (0,1))$. Thus, the continuity equation holds for the limit, i.e. $0 = \partial_t 0 + \divergence(0,f_t) = \partial_d f_t$ in $\D'((0,1)\times \R^{d-1} \times (0,1))$, i.e. $f_t(\tilde x, x_d) = f_t(\tilde x)$. By Mazur's Lemma, it follows that
\begin{equation}
\begin{aligned}
\int_0^1 \int_{\R^{d-1}}  f_t^2(\tilde x) \,d\tilde x\,dt \leq &\liminf_{\eps \to 0} \int_0^1 \int_{\R^{d-1} \times (0,1)} (W_t^\eps)_d^2\,dx\,dt\\
 \leq \alpha  \liminf_{\eps \to 0} \int_0^1 \int_{\R^{d-1} \times (0,1)} \frac{\eps^2(W_t^\eps)_d^2}{\pi_t^\eps}\,dx\,dt \leq &\alpha \liminf_{\eps \to 0} \int_0^1 \int_{\R^{d-1} \times (0,\eps)} \frac{|V_t^\eps|^2}{\rho_t^\eps}\,dx\,dt.
\end{aligned}
\end{equation}
Dividing both sides by $\alpha$ yields the part of the lower bound in the membrane $\R^{d-1} \times (0,\eps)$. For the outer part of the membrane
we find by Jensen's inequality 
\begin{equation}
\int_0^1 |W^\eps_t|(\R^d \setminus(\R^{d-1} \times (0,1))) \,dt \leq \sqrt{\int_0^1 \int_{\R^d\setminus(\R^{d-1} \times (0,1))} \left|\frac{dW_t^\eps}{d\pi_t^\eps}\right|^2 \,d\pi_t^\eps \,dt} \leq C,
\end{equation}
because the energy is finite.
Also $\int_0^1 \|W^\eps_t\|_{L^2(\R^{d-1}\times (0,1))}^2 \,dt \leq C$, from which we infer that a subsequence of $W^\eps_t$ converges vaguely to some Radon measure $W = (W_t)_{t\in [0,1]} \in \M([0,1]\times \R^d;\R^d)$ with $\partial_t \pi_t + \divergence W_t = 0$, where 
\begin{equation}
\pi_t(dx) = \begin{cases}
              \rho^+_t(dx - e_d),&\text{ if }x_d \geq 1\\
              0, &\text{ if }x_d\in (0,1)\\
              \rho_t^-(dx), &\text{ if }x_d \leq 0. 
              \end{cases}
\end{equation}
 By Lemma \ref{lma: lower semicontinuity} ($h=\infty$) we have 
\begin{equation}
\begin{aligned}
&\int_0^1  \int_{\R^d \setminus (\R^{d-1} \times (0,1))} \left|\frac{dW_t}{d\pi_t}\right|^2\,d\pi_t \,dt 
\leq &\liminf_{\eps \to 0} \int_0^1 \int_{\R^d \setminus (\R^{d-1} \times (0,1))} \left|\frac{dV_t^\eps}{d\rho_t^\eps}\right|^2\,d\rho_t^\eps\,dt.
\end{aligned}
\end{equation}
All in all we obtain
\begin{equation}
\begin{aligned}
&\int_0^1 \left( \int_{\R^d \setminus (\R^{d-1} \times (0,1))} \left|\frac{dW_t}{d\pi_t}\right|^2\,d\pi_t + \int_{\R^{d-1}} \frac{|f_t|^2}{\alpha}\,d\tilde x\right) \,dt \\
\leq &\liminf_{\eps \to 0} \int_0^1 \int_{\R^d} \left|\frac{dV_t^\eps}{d\rho_t^\eps}\right|^2\,d\rho_t^\eps\,dt.
\end{aligned}
\end{equation}
We define 
$V_t^-:=W_t|_{\R^{d-1}\times(-\infty,0]}$ and $V_t^+:=W_t(\cdot-e_d)|_{\R^{d-1}\times (0,\infty)}$. Let $\phi\in\C_c^\infty((0,1)\times\R^d_-)$ and let $\Phi\in\C^\infty_c((0,1)\times\R^{d-1}\times (-\infty,1))$ be an extension of $\phi$. Then
\begin{align}
&\int_0^1\int_{\R^d_-}\partial_t\phi_t\, d\rho_t^-\, dt
=\int_0^1\int_{\R^{d-1}\times (-\infty,1)}\partial_t\Phi_t\, d \pi_t\, dt\\
=&\int_0^1\int_{\R^{d-1}\times (-\infty,1)}\nabla\Phi_t\cdot\, d W_t\, dt\\
=&\int_0^1\int_{\R^{d-1}_-}\nabla\phi_t\cdot\, d V_t^-\, dt
+\int_0^1\int_{\R^{d-1}\times(0,1)}\partial_d\Phi_t(\tilde x,x_d)f_t(\tilde x)\, d(\tilde x,x_d)\, dt\\
=&\int_0^1\int_{\R^{d-1}_-}\nabla\phi_t\cdot\, d V_t^-\, dt
-\int_0^1\int_{\partial\R^{d}_-}\phi_t(\tilde x)f_t(\tilde x)\, d\tilde x\, dt,
\end{align}
which shows the continuity equation \eqref{eq: conteqflux} in the lower half-space. The upper half-space works similarly.
\end{proof}

To prove the upper bound, we find it is useful to represent the limit problem in Lagrangian coordinates. For curves in $W_2(\R^d)$ with finite kinetic action, this is done by the well-known superposition principle due to Smirnov \cite{smirnov1994} and applied to optimal transport in e.g. \cite{Ambrosio-Crippa}. Here, the particle trajectories may jump between the half-spaces and are thus not continuous. A natural class of curves are the special curves of bounded variation defined below, see also Figure \ref{fig: jumps}.

\begin{defi}
Given $d\in \N\setminus\{0\}$, we define the class $SBV_2^\div$ of curves in $\R^d_- \sqcup \R^d_+$ containing all $X : [0,1] \to \R^d_- \sqcup \R^d_+$ such that $X$ is absolutely continuous up to a finite jump set $J_X \subset (0,1)$, with velocity $\int_{[0,1] \setminus J_X} |\dot X_t|^2\,dt < \infty$, and mirrored traces at the jumps $X_{t^-} = SX_{t^+}$ for all $t\in J_X$, where $S:\R^d_- \sqcup \R^d_+ \to \R^d_- \sqcup \R^d_+$ is the mirror function mapping $(\tilde x, x_d)\in \R^d_\pm$ to $ (\tilde x, -x_d) \in \R^d_\mp$.

We also define the subclass $SBV_2^0$ as all curves $X\in SBV_2^\div$ with jump traces on the boundary $\partial \R^d_\pm$.

We equip $SBV_2^\div$ with the notion of weak convergence, where $X^k \rightharpoonup X$ if $X^k \to X$ in $L^1([0,1];\R^d_- \sqcup \R^d_+)$, $\dot X^k \rightharpoonup \dot X$ weakly in $L^2([0,1];\R^d)$, and the measures \\
$(\sum_{t\in J_{X^k}} \sigma(t) \delta_t)_{k\in \N} \subset \M([0,1])$ converge weakly-$\ast$ in $\M([0,1])$ to some $\nu$, with $\sum_{t\in J_{X}} \sigma(t) \delta_t = \nu|_{(0,1)}$, where $\sigma\in\{\pm 1\}$ denotes the sign of the $e_d$ component of the jump. (Here we need to exclude jumps converging to $0$ or $1$, as they vanish from the jump set)
\end{defi}

\begin{figure}[h]
\begin{center}
\includegraphics{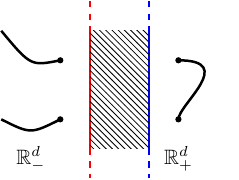}
\end{center}
\caption{A curve in the space $SBV_2^\div$. Note that if $X$ is in $SBV_2^0$ the traces at the jumps must be located on the boundary $\partial \R^d_\pm$.}\label{fig: jumps}
\end{figure}

We now state some elementary properties of $SBV_2^\div$.

\begin{lma}\label{lem: SBV properties}
The notion of weak convergence in $SBV_2^\div$ is metrizable. The underlying metric space is Polish, and $SBV_2^0$ is a weakly closed subset. Given $M>0$, the set
\begin{align}
A_M = \{X\in SBV_2^\div\,:\,|X_0|\leq M, \int_0^1 |\dot X_t|^2\,dt \leq M^2, \#  J_{X^k} \leq M^d\}
\end{align}
is weakly sequentially compact, with $SBV_2^\div = \bigcup_{M\in \N} A_M$.
\end{lma}

\begin{proof}
Since all of $L^1([0,1;\R^d_- \sqcup \R^d_+)$, $L^2([0,1];\R^d)$ with the weak topology, and $\M([0,1])$ with the weak-$\ast$ topology are metrizable and complete, these properties are inherited by $SBV_2^\div$:

If $X^k \to X$ in $L^1([0,1];\R^d_- \sqcup \R^d_+)$, $\dot X^k \rightharpoonup V \in L^2([0,1];\R^d)$, and $\sum_{t\in J_{X^k}} \sigma(t) \delta_t \weakstar \nu \in \M((0,1))$ vaguely, then $\dot X = V$, $\#  J_X = |\nu|((0,1))$, and $\sum_{t\in J_X} \sigma(t) \delta_t = \nu|_{(0,1)}$.

This shows that weak convergence in $SBV_2^\div$ is metrizable and complete. For separability, note that while $\M([0,1])$ is not separable, its subset $\{\sum_{t\in J} \sigma(t) \delta_t\,:\, J\subset[0,1]$ finite$, \sigma(t) \in \{\pm 1\}\}$ is. The fact that $SBV_2^\div = \bigcup_{M\in\N} A_M$ follows from the definition. The weak sequential compactness of $A_M$ also follows from the above argument.
\end{proof}

In the presence of a membrane, we see that some -- but not all -- particles at the membrane, will jump between the upper and lower half spaces. We model this using a stochastic jump process with rate determined by the ratio of the flux $f$ and the density of $\rho^\pm$.

\begin{prop}\label{prop: superposition}
Let $(\rho_t^-,\rho_t^+)_{t\in[0,1]} \subset \M_+(\R^d_- \sqcup \R^d_+)$ be a curve with finite limit action and finite mass, with $\partial_t \rho_t^\pm + \divergence V_t^\pm \pm f_t\Hm^{d-1}|_{\partial \R^d_\pm} = 0$. Then there exists a measure $P\in \M_+(SBV_2^0)$ with mass $P(SBV_2^0) = \rho_0^-(\R^d_-) + \rho_0^+(\R^d_+)$ such that the following hold:
\begin{itemize}
\item $X_t \sim (\rho_t^-, \rho_t^+)$ for every $t\in [0,1]$.
\item $E[\int_0^1 |\dot X_t|^2\,dt] \leq \sum_{\pm} \int_0^1 \int_{\R^d_\pm}\left|\frac{dV_t^\pm}{d\rho_t^\pm}\right|^2\,d\rho_t^\pm \,dt$.
\item The Borel measures $F_\pm\in \M_+([0,1] \times \partial \R^d_\pm)$ defined as $F_\pm(A) = E[\#\{t\in J_X\,:\,(t,X_{t^-})\in A\}]$ are absolutely continuous with respect to $dt\otimes d\Hm^{d-1}|_{\partial \R^d_\pm}$ with densities $g_t^\pm$ satisfying $g_t^\pm(\tilde x) \leq (f_t(\tilde x))^\pm$ for almost every $(t,\tilde x)$.
\end{itemize}
\end{prop}

Note that for all nonnegative measures $P\in \M_+(SBV_2^0)$ with $E[\int_0^1 |\dot X_t|^2\,dt]<\infty$ and $\sum_\pm \int_0^1 \int_{\partial \R^d_\pm} |g_t^\pm|^2\,d\Hm^{d-1}\,dt< \infty$, the laws $ (\rho_t^-, \rho_t^+)=E[\delta_{X_t}]$ have finite limit action, with $\partial_t \rho_t^\pm + \divergence V_t^\pm +  (g_t^\pm - g_t^\mp\circ S)\Hm^{d-1}|_{\partial \R^d_\pm}=0$, where $(V_t^-,V_t^+)=E[\dot X_t\delta_{X_t}]$, and
\begin{align}
 \sum_{\pm} \int_0^1 \int_{\R^d_\pm}\left|\frac{dV_t^\pm}{d\rho_t^\pm}\right|^2\,d\rho_t^\pm\,dt \leq E[\int_0^1 |\dot X_t|^2\,dt] 
\end{align}
by Jensen's inequality.

For the proof, we follow the argument in \cite[Theorem 4.4]{Ambrosio-Crippa}.

\begin{proof}
\emph{Step 1:} Instead of $(\rho_t^-, \rho_t^+)_{t\in[0,1]}$ we consider the mollified versions $\rho_t^{\pm\eps}(dx) \coloneqq \rho_t^\pm \ast \phi^{\pm\eps}(dx) + \eps e^{-|x|^2}(dx)|_{\R^d_{\pm}}$, where $\phi^{\pm\eps}\in \mathcal C_c^\infty(B(\pm \eps e_d, \eps))$ is a Dirac sequence with $\phi^{-\eps}\circ S = \phi^{+\eps}$. 

We note that after the mollification, we have $\rho_t^{\pm \eps} \in \mathcal C^\infty(\R^d_{\pm})$, Lipschitz, and strictly positive. If $\partial_t \rho_t^\pm + \divergence V_t^\pm  \pm f_t\Hm^{d-1}|_{\partial \R^d_\pm}=0$, then setting $V_t^{\pm\eps} = V_t^\pm \ast \phi^{\pm\eps}$, $v_t^{\pm\eps} = V_t^{\pm\eps}/\rho_t^{\pm\eps}$, and $g_t^{\pm\eps} = \pm f_t\Hm^{d-1}|_{\partial \R^d_{\pm}} \ast \phi^{\pm\eps}$, we have
\begin{equation}\label{eq: smooth equation}
\partial_t \rho_t^{\pm\eps} + \divergence (\rho_t^{\pm\eps} v_t^{\pm\eps}) + g_t^{\pm\eps}=0,
\end{equation}
 with $v_t^{\pm\eps}$ locally Lipschitz and satisfying the boundary values $v_t^{\pm\eps} = 0$ on $\partial \R^d_\pm$ since $V_t^{\pm\eps}=0$ on $\partial\R^d_\pm$ and $\rho_t^{\pm\eps}>0$ in $\R^d_\pm$. By Jensen's inequality and the convexity of $(V,\rho)\mapsto\frac{|V|^2}{\rho}$ in $\R^d\times(0,\infty)$ we may estimate
\begin{align}
\int_{\R^d_\pm} |v_t^{\pm\eps}|^2\,d\rho_t^{\pm\eps}\,dt \leq  \int_{\R^d_\pm} \left|\frac{d V_t^{\pm}}{d\rho_t^\pm}\right|^2\,d\rho_t^{\pm}\,dt. 
\end{align}

We note that $g_t^{\pm\eps}$ is no longer supported on the boundary but in a $2\eps$-neighborhood of the same.

We now define a random curve $X \in SBV_2^\div$. First, its starting point $X_0\in \R^d_- \sqcup \R^d_+$ is distributed according to $(\rho_0^{-\eps},\rho_0^{+\eps})$. Independently of the starting point, take a random realization of the $1$-Poisson process, yielding discrete times $\cT = \{t_i\}_{i\in \N} \subset [0,\infty)$. Then define the random curve $(X_t,T_t)\colon [0,1]\to \R^d_-\sqcup \R^d_+\times [0,\infty)$ as the solution to the ODE
 \begin{align}
 \begin{cases}
 X_0 = X_0\\
T_0 = 0\\
 \dot X_t = v^{\sigma(T_t)\eps}_t (X_t)\\
 \dot T_t = (g_t^{\sigma(T_t)\eps}(X_t))_+/\rho_t^{\sigma(T_t)\eps}(X_t).
 \end{cases}
 \end{align}
 
 Here $\sigma:[0,\infty) \to \{-1,1\}$ is the function indicating whether $X_t$ is in the lower or upper half-space, with $\sigma(0)$ determined by the starting half-space of $X_0$ and jump set $J_\sigma = \cT$. Clearly $X\in SBV_2^\div$ almost surely. In particular, if $X_t$ is in the lower half-space, it jumps to the upper half-space whenever $T_t=t_i$ and vice versa. Because its derivative is nonnegative, $T_t$ is nondecreasing.
 
 By using It\^o's formula for semimartingales with jumps (Section 2.1 in \cite{protter}) we see that the distribution $X_t \sim (\mu_t^{-\eps},\mu_t^{+\eps})$ solves the Cauchy problem
 \begin{align}
 \begin{cases}
 \mu_0^{\pm\eps} = \rho_0^{\pm\eps}\\
 \partial_t \mu_t^{\pm\eps}  + \divergence (\mu_t^{\pm\eps} v_t^{\pm\eps}) +  \frac{(g_t^{\pm\eps})_+}{\rho_t^{\pm\eps}}\mu_t^{\pm\eps} - \frac{(g_t^{\mp\eps})_+}{\rho_t^{\mp\eps}}\mu_t^{\mp\eps} \circ S  = 0, 
 \end{cases}
 \end{align}
 as does $(\rho_t^{-\eps},\rho_t^{+\eps})$, since $(g_t^{\pm\eps})_+ - (g_t^{\mp\eps})_+\circ S = g_t^{\pm\eps}$, and $(\rho_t^{-\eps},\rho_t^{+\eps})$ solves \eqref{eq: smooth equation}. Because the solution is unique by the Cauchy-Kovalevskaya theorem, we have $\mu_t^{\pm\eps} = \rho_t^{\pm\eps}$ for every $t\in[0,1]$. We take $P^\eps\in \M_+(SBV_2^\div)$ to be the law of $X$. Defining for a Borel $A\subset [0,1]\times \R^d_\pm$ the nonnegative measure $F^{\pm\eps}(A) \coloneqq E^\eps[ \#  \{t\in J_X\,:\,(t,X_{t^-})\in A\}]$, we note that $F^{\pm\eps}$ is absolutely continuous with respect to $dt\otimes dx$ with density $g_t^{\pm\eps}(x)$ according to the construction.\\
\emph{Step 2:}
The next step is to show that the $P^\eps \in \M_+(SBV_2^\div)$ are tight. To this end we use the weakly sequentially compact sets $A_M$ from Lemma \ref{lem: SBV properties} and show that $\lim_{M \to \infty} \sup_{\eps > 0} P^\eps(SBV_2^\div \setminus A_M) = 0$. We check each of the three conditions defining $A_M$:

 \begin{align}
 \sup_{\eps > 0}P^\eps(|X_0| > M) = \sup_{\eps > 0} \rho^\eps_0(\R^d_- \setminus \overline{B(0,M)} \sqcup \R^d_+ \setminus \overline{B(0,M)} )\to_{M\to \infty} 0,
 \end{align}
 since the $(\rho^\eps_0)_{\eps >0}\subset \M_+(\R^d_- \sqcup \R^d_+)$ are tight, where $\rho_0^\eps=(\rho_0^{-\eps},\rho_0^\eps)$.

 For the second condition, this follows from the finity of the transport part of the energy and Markov's inequality:
 
 \begin{align}
 \sup_{\eps > 0} P^\eps(\int_0^1 |\dot X_t|^2\,dt > M^2) \leq \sup_{\eps > 0} \frac{1}{M^2} E^\eps[\int_0^1 |\dot X_t|^2\,dt] \to_{M\to \infty} 0.
 \end{align}

 For the third condition, this follows from the finity of the membrane part of the energy and H\"older's and Markov's inequalities. In order to use H\"older's inequality, we note that if $|X_0|\leq M$ and $\int_0^1 |\dot X_t|^2 \,dt\leq M^2$, then $|X_t|\leq 2M$ for all $t$, independently of the jump set. Thus,
 
 \begin{align}
 \begin{aligned}
&\sup_{\eps \in (0,1)} P^\eps(\sup_t |X_t|\leq 2M, \#  J_X > M^d) \leq \sup_{\eps > 0} \frac{1}{M^d}E^\eps[\#  J_X \1_{\sup_t |X_t| \leq 2M}]\\
 \leq & \sup_{\eps \in (0,1)}\frac{1}{M^d} \int_0^1 \int_{ B(0,2M)} |g_t^{-\eps}( x)| + |g_t^{+\eps}( x)|\,d x\, dt\\
 \leq & \sup_{\eps \in (0,1)}\frac{2}{M^d} \int_0^1 \int_{\tilde B(0,2M+\eps)} |f_t|(\tilde x)\,d\tilde x\, dt\\
 \leq &\sup_{\eps \in (0,1)} C(d)\frac{ (M+\eps)^{(d-1)/2}}{M^d} \left(\int_0^1 \int_{\tilde B(0,2M+2\eps)} f_t^{2}(\tilde x) \,d\tilde x\, dt\right)^{1/2} \to_{M\to \infty} 0.
 \end{aligned}
 \end{align}
 
 This shows that the $(P^\eps)_{\eps > 0} \subset \M_+(SBV_2^\div)$ are weakly tight, so that by Prokhorov's theorem they have a weakly convergent subsequence $P^\eps \weakstar P\in \M_+(SBV_2^0)$. It is easily seen  that the law of $X_t$ under $P$ is $X_t\sim (\rho_t^-,\rho_t^+)$. Because the pathwise energy is weakly-$\ast$ lower semicontinuous, it follows from the Portmanteau theorem
 \begin{align}
 E\left[\int_0^1 |\dot X_t|^2\,dt\right] \leq \liminf_{\eps \to 0} E^\eps\left[\int_0^1 |\dot X_t|^2\,dt\right].
 \end{align}
 For the membrane part, note that as $P^\eps \weakstar P$, we have for any relatively open $A\subset [0,1]\times \R^d_\pm$ that
 \begin{align}
 F^\pm(A) \coloneqq E[\# \{t\in J_X\,:\,(t,X_{t^-})\in A\}] \leq \liminf_{\eps\to 0} F^{\pm\eps}(A),
 \end{align}
 since $X\mapsto \# \{t\in J_X\,:\,(t,X_{t^-})\in A\}$ is weakly sequentially lower semicontinuous. On the other hand, clearly $F^{\pm\eps} \weakstar(\pm f_t(\tilde x))_+ dt\otimes d\Hm^{d-1}|_{\partial \R^d_\pm}(\tilde x)$, so that $F^\pm$ is absolutely continuous with density at most $(\pm f_t(\tilde x))_+$.

\end{proof}

\begin{eg}
Take $\rho^\pm = \frac{1}{\omega_{d-1}R^{d-1}} \Hm^{d-1}|_{\partial \R^d_\pm \cap B(0,R)}$ and take $\rho_t^- = (1-t)\rho^-$, $\rho_t^+ = t \rho^+$. Then $V_t^\pm = 0$, $f_t(\tilde x) = \frac{1}{\omega_{d-1}R^{d-1}} \1_{B(0,R)}(\tilde x)$ in the continuity equation. The probability measure $P\in \P(SBV_2^0)$ is the uniform distribution on the curves
\begin{align}
X_t = \begin{cases} x_0, &\text{ if }0\leq t \leq t_0\\ Sx_0, &\text{ if }t_0 < t \leq 1\end{cases} 
\end{align}
for $t_0\in [0,1], x_0\in \partial \R^d_- \cap B(0,R)$. We see that there is no way to choose the jump times deterministically.
\end{eg}

We now use this Lagrangian representation to prove the upper bound. Roughly, instead of having particles teleport across the membrane of width $\eps >0$, we replace a particle entering the membrane from one side with a different one exiting on the other side. This technique is inspired by the magical illusion ``The Tranported Man" from the novel \textit{The Prestige}\cite{prestige}, where instead of actually teleporting himself, a magician simply exits the stage as his identical twin brother enters it at the same time. The illusion is the difference between the Eulerian and the supposed Lagrangian formulation of the transport. 

\begin{proof}[Proof of the upper bound]
Take a curve with finite limit action $(\rho_t^-, \rho_t^+)_{t\in [0,1]}$ and represent it using $P\in \M_+(SBV_2^0)$ as in Proposition \ref{prop: superposition}. We shall modify these paths to pay heed to the finite thickness of the membranes. To this end, we modify the curves in $\supp P$ as follows:

For any measure $F^\pm \in \M_+([0,1]\times \partial \R^d_\pm)$ with absolutely continuous density $dF^\pm = f_t^\pm(\tilde x)\,(dt \otimes d\Hm^{d-1}(\tilde x))$, define a stopping time $\tau_F:SBV_2^0 \to [0,\infty]$ through
\begin{equation}
\tau_F(X) \coloneqq \inf\left\{t\in J_X\,:\, \int_0^t f_s^-(\tilde X_t) + f_s^+(\tilde X_t)\,ds < \alpha \eps^2\right\}.
\end{equation}
Note that $\tau_F$ is Borel-measurable and decreasing in $F$.

On the other hand, given a measurable stopping time $\tau:SBV_2^0 \to [0,\infty]$, define a Borel measure $F_\tau^\pm\in \M_+([0,1]\times \partial \R^d_\pm)$ through
\begin{equation}
F_\tau^\pm(A) \coloneqq E[\#\{t\in J_X\,:\,t\leq \tau(X), (t,X_{t^-})\in A\}],
\end{equation}
where the expectation is taken with respect to $P$. Note that $F_\tau$ is decreasing in $\tau$ and $F_\tau^\pm \leq F^\pm$. In particular, $F_\tau^\pm$ is absolutely continuous.

Define for every $\eps>0$ first $\tau_0 \coloneqq \infty$, then $F_k^\pm \coloneqq F_{\tau_k}^\pm\in \M_+([0,1]\times \partial \R^d_\pm)$ and $\tau_{k+1} \coloneqq \tau_{F_k}$. Then $(\tau_k)_{k\in\N} \subset [0,\infty]^{SBV_2^0}$ forms a nonincreasing sequence of Borel-measurable stopping times and converges pointwise to some Borel-measurable $\tau_\eps:SBV_2^0\to[0,\infty]$. Likewise, $(F_k^\pm)_{k\in\N} \subset \M_+([0,1]\times \partial \R^d_\pm)$ forms a nonincreasing sequence and converges weakly-$\ast$ to a a limit measure $F^{\pm\eps} \in \M_+([0,1]\times \partial \R^d_\pm)$ with density $f_t^{\pm\eps}(\tilde x) \leq f_t^\pm(\tilde x)$ for  every $(t,\tilde x)\in [0,1]\times \partial \R^d_\pm$. By continuity, we then have
\begin{equation}\label{eq: stopping time}
\tau_\eps(X) = \inf\{t\in J_X\,:\, \int_0^t f_s^{-\eps}(\tilde X_t) + f_s^{+\eps}(\tilde X_t)\,ds < \alpha \eps^2\}
\end{equation}
for $P$-almost every $X\in SBV_2^0$, and
\begin{equation}\label{eq: flux measure}
F^{\pm\eps}(A) = E[\#\{t\in J_X\,:\,t \leq \tau_\eps(X), (t,X_{t^-})\in A\}].
\end{equation}

Now we define the stopped process $SBV_2^0 \ni X \mapsto X^\eps \in SBV([0,1];\R^d)$ through
\begin{equation}
X^\eps_t \coloneqq \begin{cases}
X_t + \eps e_d, &\text{ if }t\leq\tau_\eps(X), X_t\in \R^d_+\\
X_{\tau_\eps(X)^-}+(\eps - \frac{1}{\alpha\eps}\int_0^{\tau_\eps(X)} f^{+\eps}_s(\tilde X_{\tau_\eps(X)^-})\,ds )e_d, &\text{ if }t > \tau_\eps(X), X_{\tau_\eps(X)^-}\in \R^d_+\\
X_{\tau_\eps(X)^-}+ \frac{1}{\alpha\eps}\int_0^{\tau_\eps(X)} f^{-\eps}_s(\tilde X_{\tau_\eps(X)^-})\,ds\, e_d, &\text{ if }t > \tau_\eps(X), X_{\tau_\eps(X)^-}\in \R^d_-\\
X_t,           &\text{ if }t\leq\tau_\eps(X), X_t\in \R^d_-.
\end{cases}
\end{equation}

\begin{figure}
\begin{center}
\includegraphics{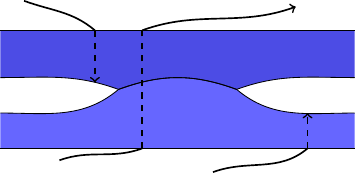}
\caption{The stopped curves $X_t^\eps$ end in the membrane the first time they try to cross at a point where the membrane is not yet filled, increasing the size of the filled region $U_t^\eps$.}\label{fig: membrane}
\end{center}
\end{figure}

We note that in the stopped process, the first $\alpha \eps^2$ particles to attempt a jump across the membrane at $\tilde x$ are instead frozen inside the membrane. This allows us to easily construct the recovery sequence as follows:
\begin{equation}
(\rho_t^\eps)_{t\in [0,1]} \coloneqq (E[\delta_{X^\eps_t}])_{t\in [0,1]} \subset \M_+(\R^d).
\end{equation}

We define the momentum field first as a measure $V\in \M([0,1]\times \R^d;\R^d)$ and later show that $V$ is absolutely continuous in time:
\begin{equation}
V^\eps \coloneqq E[\dot X_t^\eps (dt \otimes \delta_{X_t^\eps})] + E\left[\sum_{t\in J_{X_t^\eps}} \frac{[X_t^\eps]}{|[X_t^\eps]|} (\delta_t \otimes \Hm^1|_{[X_{t^-}^\eps,X_{t^+}^\eps]})\right].
\end{equation}
Here $[X_t^\eps]\in \R^d$ denotes the jump of $X_t^\eps$, which is always parallel to $e_d$.

By the linearity of the continuity equation, it is clear that $\partial_t \rho_t^\eps + \divergence V^\eps = 0$ in $\D'((0,1)\times \R^d)$.

Define $U_t^\eps \subset \R^{d-1}\times (0,\eps)$ as the set
\begin{align}
U_t^\eps = \left\{(\tilde x, x_d)\in \R^{d-1} \times (0,\eps)\,:\,x_d \leq \frac{1}{\alpha \eps} \int_0^t f_s^{-\eps}(\tilde x)\,ds\text{ or }\eps-x_d \leq \frac{1}{\alpha\eps}\int_0^t f_s^{+\eps}(\tilde x)\,ds\right\}.
\end{align}

We claim that $\rho_t^\eps|_{\R^{d-1}\times (0,\eps)} = \alpha \eps \Lm^d|_{U_t^\eps}$, see Figure \ref{fig: membrane}. We test this against cylindrical sets $\tilde A \times I$, with $\tilde A \subset \R^{d-1}$ and $I\subset (0,\eps)$ Borel:
\begin{equation}
\begin{aligned}
&\rho_t^\eps(\tilde A \times I)\\
 = & P(\tau_\eps(X)<t,\tilde X_{\tau_\eps(X)} \in \tilde A, X^\eps_{\tau_\eps(X)^+}\cdot e_d\in I)\\
= & F^{-\eps} \left(\left\{(s,\tilde x)\,:\, s\leq t, \tilde x \in \tilde A,\frac{1}{\alpha\eps} \int_0^s f_r^{-\eps}(\tilde x)\,dr \in I, \int_0^s f_r^{-\eps}(\tilde x) + f_r^{+\eps}(\tilde x)\,dr < \alpha \eps^2 \right\} \right)\\
& + F^{+\eps} \left(\left\{(s,\tilde x)\,:\, s\leq t, \tilde x \in \tilde A,\frac{1}{\alpha\eps} \int_0^s f_r^{-\eps}(\tilde x)\,dr \in \eps - I, \int_0^s f_r^{-\eps}(\tilde x) + f_r^{+\eps}(\tilde x)\,dr < \alpha \eps^2 \right\} \right)\\
= & \int_{\tilde A} \int_0^t \left(f_s^{-\eps}(\tilde x) \mathds{1}_{\{\frac{1}{\alpha\eps}\int_0^s f_r^{-\eps}(\tilde x)\,dr \in I\}} + f_s^{+\eps}(\tilde x) \mathds{1}_{\{\frac{1}{\alpha\eps} \int_0^s f_r^{+\eps}(\tilde x)\,dr \in \eps - I\}} \right) \mathds{1}_{\{\int_0^s f_r^{-\eps}(\tilde x) + f_r^{+\eps}(\tilde x)\,dr < \alpha \eps^2\}}\,ds\,d\tilde x\\
 = &\alpha \eps\Lm^d((\tilde A \times I) \cap U_t^\eps).
 \end{aligned}
\end{equation}
Here we used \eqref{eq: stopping time}, \eqref{eq: flux measure}, Fubini's theorem, and the change of variables formula. The claim is shown. In particular, $\rho_t^\eps \leq h^\eps$. It follows that $\rho_t^\eps|_{\R^{d-1}\times (-\infty,0]} \weakstar \rho_t^-$ and $\rho_t^\eps|_{\R^{d-1}\times (0,\infty))} \weakstar \rho_t^+$ for every $\eps > 0$, since $P(\tau_\eps(X)<\infty, |X_1^\eps| < R) \leq C(d)R^{d-1}\alpha \eps^2$, whereas by Prokhorov's theorem $P(|X_1^\eps| > R) \to 0$ as $R\to \infty$. All in all, $|\rho_t^\eps - \rho_t| \to 0$.

Finally, we have to estimate the action. Outside of the membrane, this is simply Jensen's inequality:
\begin{equation}\label{eq: outside energy}
\int_0^1 \int_{\R^d \setminus (\R^{d-1} \times (0,\eps))} \left|\frac{dV_t^\eps}{d\rho_t^\eps}\right|^2\,d\rho_t^\eps\,dt \leq E\left[\int_0^1 |\dot X_t|^2\,dt\right].
\end{equation}

Inside the membrane, we first note that $V^\eps$ is absolutely continuous with respect to $(t,x)$, with density
\begin{equation}
V^\eps(dt \otimes dx) = \mathds{1}_{U_t^\eps}(x)(f_t^{+\eps}(\tilde x) - f_t^{-\eps}(\tilde x))e_d\,(dt \otimes dx),
\end{equation}
so that
\begin{equation}\label{eq: inside energy}
\int_0^1 \int_{\R^{d-1} \times (0,\eps)} \frac{|V_t^\eps(x)|^2}{\rho_t^\eps(x)}\,dx\,dt = \int_0^1\int_{U_t^\eps}\frac{(f_t^{+\eps}(\tilde x) - f_t^{-\eps}(\tilde x))^2}{\alpha \eps}\,dx\,dt \leq \frac{1}{\alpha}\int_0^1 \int_{\R^{d-1}} f_t^2(\tilde x)\,d\tilde x\,dt. 
\end{equation}

Combining \eqref{eq: outside energy} with \eqref{eq: inside energy} yields the upper bound
\begin{equation}
\int_0^1 \int_{\R^d} \left|\frac{dV_t^\eps}{d\rho_t^\eps}\right|^2\,d\rho_t^\eps\,dt \leq E_0((\rho_t^-,\rho_t^+)_{t\in[0,1]}).
\end{equation}

\end{proof}

\begin{eg}
Consider $d=1$ and set $\rho_t^- = (1-t)\delta_0$ and $\rho_t^+ = t\delta_0$. This is an optimal curve connecting its two end points. The flux is $f(t) = 1$ and the cost is simply $\frac{1}{\alpha}$.

Another optimal curve is given by $\rho_t^- = \mathds{1}_{[-1+t,0]}$, $\rho_t^+ = \mathds{1}_{[0,t]}$. This curve is also optimal, with the same flux $f(t) = 1$ and cost $1 + \frac{1}{\alpha}$.

For $d>1$, the optimal curve between $\rho_0 = (\delta_0, 0)$ and $\rho_1 = (0,\delta_0)$ is supported on the curves
\begin{equation}
X_t = \begin{cases}
      \frac{t}{t_0}\tilde x_0 &\text{, if }t\leq t_0\\
      \frac{1-t}{1-t_0}S\tilde x_0 &\text{, if }t>t_0,
      \end{cases}
\end{equation}
with $t_0\in[0,1]$ and $\tilde x_0 \in \partial \R^d_-$. The distribution $\mu(dt_0,d\tilde x_0) \in\P([0,1]\times \partial \R^d_-)$ of crossing coordinates $(t_0,\tilde x_0)$ then minimizes
\begin{equation}
E\left[|\tilde x_0|^2 \left(\frac{1}{t_0}+\frac{1}{1-t_0}\right)\right] + \int_0^1 \int_{\partial \R^d_-} \frac{1}{\alpha}\left(\frac{d\mu}{d(t_0,\tilde x_0)}\right)^2\,d\tilde x_0\,dt_0
\end{equation}
among all probability measures. A simple calculation shows that
\begin{equation}
\frac{d\mu}{d(t_0,\tilde x_0)} = \left(c(d,\alpha) - \frac{\alpha}{2}|\tilde x_0|^2\left(\frac{1}{t_0}+\frac{1}{1-t_0}\right)\right)_+\, ,
\end{equation}
with $c(d, \alpha) > 0$ chosen uniquely so that $\mu$ is a probability measure. We note that for $d=1$, the only crossing point is $\tilde x_0 = 0$, and we recover $\frac{d\mu}{dt_0} = 1$.
\end{eg}

\section{Homogenization}\label{section: homogenization}
In this section we prove Theorem \ref{thm:homogenizationintro}, i.e. we will show that $E_{h_\eps}$ $\Gamma$-converges to $E_{\hom}$.

 Let us start by collecting a few properties of the functional $E_{\hom}$ defined in \eqref{eq: homogenized}.
\begin{lma}\label{lma: prophom}
The following properties hold:
\begin{enumerate}
\item [(i)] For all $m\in(0,\int_{\T^d} h(x)\, dx]$ there exist minimizers $\nu(m,U)\in L^1(\T^d)$ and
$W(m,U)\in L^2(\T^d;\R^d)$ of $f_{\hom}(m,U)$.
\item [(ii)] The map $(m,U)\mapsto f_{\hom}(m,U)$ is convex, lower semicontinuous, and $2$-homogeneous in $U$.
\item[ (iii)] $E_{\hom}$ is convex and lower semicontinuous.
\item [(iv)] There is a constant $C$ depending only on $\{h>0\}$ and $\alpha$ such that $\frac{|U|^2}{m} \leq f_{\hom}(m,U) \leq C \frac{|U|^2}{m}$ for all $U\in \R^d$, $m\in (0,\int_{\T^d} h(x)\,dx]$.
\item [(v)] If $m\leq \inf_{\T^d} h$ then $f_{\hom}(m,U)=\frac {|U|^2}{m}$.
\item [(vi)] With $C$ as above,
\begin{align}\label{eq: lipschi}
f_{\hom}(m,U+Z) \leq f_{\hom}(m,U) + C\frac{|U+Z||Z|}{m} ,
\end{align}
 for all $U,Z\in \R^d$, $m\in (0,\int_{\T^d} h(x)\,dx]$. In addition, $f_{\hom}$ is locally Lipschitz in $(0,\int_{\T^d}h(x)\,dx]\times \R^d$.
\end{enumerate}
\end{lma}
\begin{proof}
$(i)$ First note that $f_{\hom}(m,U)\geq 0$. We take minimizing sequences $(\nu^n)_{n\in \N} \subset L^\infty(\T^d)$, $(W_n)_{n\in \N} \subset L^2(\T^d;\R^d)$. Then $0\leq \nu^n(x) \leq h(x)$ almost everywhere and $\int \nu^n(x)\, dx=m$. By the Banach-Alaoglu Theorem there exists a subsequence $\nu^n$ converging weakly-$\ast$ in $L^\infty(\T^d)$ to some $\nu$ satisfying $0\leq \nu(x) \leq h(x)$ almost everywhere and $\int \nu(x)\, dx=m$. Since $\int_{\T^d} |W_n(x)|^2\,dx \leq \frac{1}{\alpha} \int \frac{|W_n(x)|^2}{\nu_n(x)}\,dx$, we also get a subsequence $W_n \rightharpoonup W$ in $L^2(\T^d;\R^d)$, with $\divergence W = 0$ in $\D'(\T^d)$ and $\int_{\T^d} W(x)\,dx = U$. By the convexity and lower semicontinuity of the function $(m, U) \mapsto \frac{|U|^2}{m}$ and Mazur's Lemma, we have
\begin{align}
\int_{\T^d} \frac{|W(x)|^2}{\nu(x)}\,dx \leq \liminf_{n \to \infty} \int_{\T^d} \frac{|W_n(x)|^2}{\nu_n(x)}\,dx,
\end{align}
which shows that $(\nu, W)$ are minimizers.

$(ii)$ These properties are inherited from the function $(\nu, W) \mapsto \frac{|W|^2}{\nu}$.

$(iii)$ This is the result of Lemma \ref{lma: lower semicontinuity}.

$(iv)$ The lower bound follows from Jensen's inequality. For the upper bound, consider the vector field $X_U\in L^2(\{h>0\};\R^d)$ from Lemma \ref{lma: potential flow} below, and find $\nu \in L^1(\T^d)$ such that $h(x) \geq \nu(x) \geq \min(m, \alpha)$ almost everywhere in $\{h>0\}$, and $\int_{\T^d} \nu(x) \,dx = m$. Then
\begin{align}
\int_{\T^d} \frac{|X_U(x)|^2}{\nu(x)}\,dx \leq \int_{\{h>0\}} \frac{|X_U(x)|^2}{\min(m,\alpha)}\,dx \leq C(h) \frac{|U|^2}{m}.
\end{align}

$(v)$ The lower bound is shown in $(iv)$. For the upper bound, take $\nu(x) = m$ and $W(x) = U$.

$(vi)$ This follows from $(ii)$ and $(iv)$: Let $p\in \partial_U^-f_{\hom}(m,U)$. Then
\begin{align}
C\frac{|U|^2}{m} \geq f_{\hom}(m,U+\frac{|U|}{|p|}p) - f_{\hom}(m,U) \geq |p||U|,
\end{align}
so that $|p|\leq C\frac{|U|}{m}$.

Now take $U,Z\in \R^d$, $m\in (0,\int_{\T^d} h(x)\,dx]$, $p\in \partial_U^-f_{\hom}(m, U+Z)$. Then
\begin{align}
f_{\hom}(m, U+Z) \leq f_{\hom}(m, U) - p\cdot Z \leq f_{\hom}(m, U) + C\frac{|U+Z||Z|}{m},
\end{align}
which is \eqref{eq: lipschi}. In addition, for $0<m_1<m_2\leq\int_{\T^d} h(x) \, dx$, we have
\begin{align}
0\geq f_{\hom}(m_2,U)-f_{\hom}(m_1,U)\geq f_{\hom}(m_1,U)\frac{m_1-m_2}{m_2}.
\end{align}

To see the first inequality, start with a minimizer $\nu_1 = \nu(m_1,U), W_1 = W(m_1,U)$. Then $(\nu_1 + (m_2 - m_1) \frac{h-\nu_1}{h-m_1}, W_1)$ is a competitor for $f_{\hom}(m_2, U)$.

To see the second inequality, start with a minimizer $\nu_2 = \nu(m_2,U), W_2 = W(m_2,U)$. Then $(\frac{m_1}{m_2} \nu_2, W_2)$ is a competitor for $f_{\hom}(m_1,U)$.

Together with the growth condition (iv) we obtain the local Lipschitz property on $(0,\int_{\T^d} h(x)\, dx]\times \R^d$.
\end{proof}

The following lemma turns out to be crucial.

\begin{lma}\label{lma: potential flow}
There is a constant $C>0$ depending only on $\{h>0\}$ such that for every $U\in \R^d$ there is a vector field $X_U\in \mathcal C_c^\infty (\T^d \cap \{h>0\};R^d)$ such that $\divergence X_U = 0$ in $\D'(\T^d)$, $\int_{\{h>0\}} X_U(x)\,dx = U$, and $\int_{\{h>0\}} |X_U(x)|^2\,dx \leq C|U|^2$.
\end{lma}

\begin{proof}
Let $\gamma:[0,1] \to \R^d$ be a Lipschitz curve. Define the vector-valued measure $M \coloneqq \gamma_\# (\dot \gamma\, dt) \in \M(\R^d;\R^d)$. Then $\divergence M = \delta_{\gamma_1}-\delta_{\gamma_0}$ in $\D'(\R^d)$, $ M(\R^d) = \gamma(1) - \gamma(0)$, and $|M|(\R^d) \leq \int_0^1 |\dot \gamma(t)|\,dt = L(\gamma)$.

Let $x\in \{h>0\}$. By the conditions on $\{h>0\}$, there are Lipschitz curves $\gamma_j:[0,1]\to \{h>0\}$, $j=1,\ldots,d$, such that $\gamma_j(0) = x$, $\gamma_j(1) = x+e_j$, and $\delta \coloneqq \min_j \dist(\gamma_j, \partial\{h>0\}) > 0$.

Define $X_U = \sum_{z\in \Z^d} \sum_{j=1}^d U_j ((\gamma_j-z)_\# \dot \gamma_j) \ast \phi_\delta \in \mathcal C_c^\infty(\{h>0\};\R^d)$, where $\phi_\delta \in \mathcal C_c^\infty(B(0,\delta))$ is a standard mollifier. We note that $X_U$ is $\Z^d$-periodic, $\int_{[0,1)^d} X_U(x)\,dx = U$, $\divergence X_U = 0$, and $\|X_U\|_{L^2([0,1)^d)} \leq C \|\phi_\delta\|_{L^2} \sum_{j=1}^d |U_j| L(\gamma_j)^{d+1} \leq C(h) |U|$, where we used Young's convolution inequality and the finite overlap of the curves $(\gamma_j - z)_{z\in \Z^d, j=1,\ldots,d}$. The projection of $X_U$ to $\T^d$ inherits all the relevant properties.  
\end{proof}

We need the following lemma to estimate corrector errors.

\begin{lma}[Local Poincar\'e-trace inequality]\label{lma: poincare-trace}
There are constants $R>0$, $C>0$ depending only on $\{h > 0\}$ such that for any $\eps>0$, $a\in\R^d$, $u\in H^1_{\loc}(\{h_\eps > 0\})$, we have
\begin{equation}
\int_{(a+[0,\eps]^d)\cap \{h_\eps > 0\}} (u-\overline{u})^2 \,dx + \eps \int_{\partial(a+[0,\eps]^d) \cap \{h_\eps > 0\}} (u-\overline{u})^2 \,d\Hm^{d-1} \leq C \eps^2 \int_{a + [-R\eps,R\eps]^d \cap \{h_\eps > 0\}} |\nabla u|^2\,dx.
\end{equation}
Here $\overline u = \fint_{a+[0,\eps]^d} u \,dx$.
\end{lma}

This differs from the standard Poincar\'e-trace inequality (see e.g. Theorem 12.3 in \cite{LeoniBook}) in that the smaller cube is not connected, nor is either cube Lipschitz-bounded.

\begin{figure}[h]
\begin{center}
\includegraphics[scale=0.5]{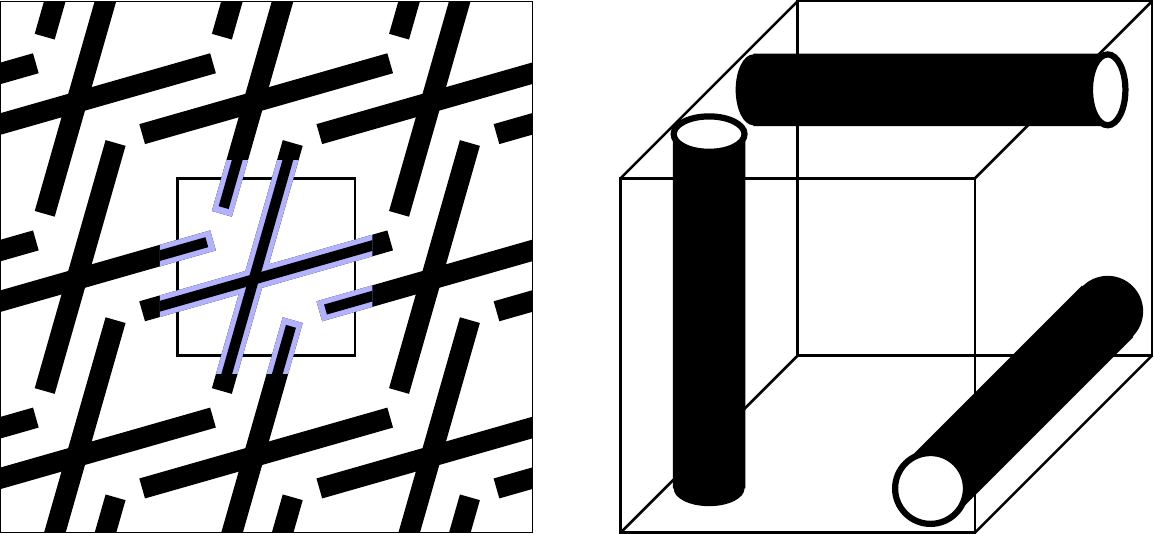}
\caption{Left: Even though the unit cell is not connected we control the $L^2$ variation through the $L^2$-norm of the gradient in the larger cell. Right: If $\{h>0\}$ is not connected Theorem \ref{thm:homogenizationintro} fails. In this case mass may only move in the coordinate directions.}\label{fig: poincare}
\end{center}
\end{figure}

\begin{proof}
The statement is independent of $\eps$. We only have to show it for $\eps = 1$ and $a\in [0,1]^d$.

We take $R>3$ as any number such that all $y,y'\in [0,2]^d \cap \{h>0\}$ are connected by a rectifiable path in $(-R,R)^d \cap \{h>0\}$.

Assume that for this choice of $R$, no such $C$ exists. Then there exists a sequence $(u_n)_{n\in \N} \subset H^1([-R,R]^d)$ and a sequence $(a_n)_{n\in \N} \subset [0,1]^d$ such that $\int_{a_n+[0,1]^d} u_n\,dx = 0$, $\int_{a_n+[0,1]^d} u^2\,dx + \int_{\partial(a_n + [0,1]^d)} u^2 \,d\Hm^{d-1} = 1$, and $\int_{[-R,R]^d} |\nabla u_n|^2 \to 0$.

Because $\{h>0\}$ is Lipschitz bounded, we can cover $\partial\{h > 0\} \cap [0,2]^d$ with finitely many open rectangles $(R_i)_{i\in I}$ such that, up to a rigid motion, $\{h>0\} \cap R_i = \{(\tilde y, y_d)\,:\,\tilde y \in \tilde R_i, 0<y_d<f_i(\tilde y)\}$, where $f:\R^{d-1} \to (0,\infty)$ is Lipschitz. 

From \cite[Theorem 12.3]{LeoniBook}, we infer that there exists a bounded linear extension operator $E:H^1([-R,R]^d\cap \{h>0\}) \to H^1(([-R,R]^d \cap \{h>0\}) \cup \bigcup_{i\in I}R_i)$ such that $Eu = u$ almost everywhere in $[-R,R]^d \cap\{h>0\}$ and
\begin{equation}
\int_{\bigcup_{i \in I} R_i} |\nabla Eu|^2\,dx \leq C \int_{[-R,R]^d \cap \{h>0\}} |\nabla u|^2\,dx.
\end{equation}

Note that only $\nabla u$ appears on the right-hand side since we are not looking for a global extension.

Extending each $u_i$ using this operator, we extract a subsequence (not relabeled) such that $a_i \to a$, $Eu_i \to u$ in $L^2_{\loc}(([-R,R]^d \cap \{h>0\}) \cup \bigcup_{i\in I} R_i)$, and $\nabla Eu_i \to 0$ in $L^2(([-R,R]^d \cap \{h>0\}) \cup \bigcup_{i\in I} R_i)$. Also, the traces $u_n|_{\partial(a_n+[0,1]^d)}\mathds{1}_{\{h>0\}}$ converge in $L^2(\partial([0,1]^d))$ to the trace $u|_{\partial(a+[0,1]^d)}\mathds{1}_{\{h>0\}}$.

It follows that $u$ is piecewise constant. Because any two points in $[0,2]^d$ are path-connected in the domain, $u$ is constant in $[0,2]^d\cap \{h>0\}$. Because $\int_{a+[0,1]^d} u \,dx = 0$, $u=0$ almost everywhere in $[0,2]^d \cap \{h>0\}$. However, we have
\begin{equation}
\int_{(a+[0,1]^d)\cap \{h>0\}} u^2 \,dx + \int_{\partial(a+[0,1]^d)\cap \{h>0\}} u^2 \,d\Hm^{d-1} = 1,
\end{equation}
a contradiction.
\end{proof}
We note that this implies the usal Poincar\'e-trace inequality in particular for $\eps\Z^d$-periodic functions in $H^1(\{h_\eps > 0\})$. 

Finally we prove Theorem \ref{thm:homogenizationintro}.
\subsection{Proof of Theorem \ref{thm:homogenizationintro}}

\begin{proof}[Proof of the lower bound]
We start with a sequence of curves $(\rho_t^\eps)_{t\in [0,1]} \subset L^\infty (\T^d)$ with $0 \leq \rho_t^\eps(dx) \leq h_\eps(x)\,dx$, together with a sequence of momentum fields $(V_t^\eps)_{t\in[0,1]} \subset L^2(\T^d;\R^d)$, such that $\partial_t \rho_t^\eps + \divergence V_t^\eps = 0$ in $\D'((0,1)\times\T^d)$ and
\begin{equation}
E_{h_\eps}((\rho_t^\eps)_t) = \int_0^1 \int_{\T^d} \frac{|V_t^\eps|^2}{\rho_t^\eps}\,dx\,dt \leq C < \infty.
\end{equation}
\\
\emph{Step 1:}
We consider instead averaged versions $(\rho_t^{\delta,\eps})_t, (V_t^{\delta, \eps})_t$ defined through
\begin{equation}\label{eq: average lower bound}
\begin{aligned}
\rho_t^{\delta,\eps} (x) = &(1- \delta)\sum_{z\in \eps\Z^d / \Z^d} H_\delta^\eps(z)\fint_{t-\delta}^{t+\delta} \rho_s(x-z)\,ds+ \delta \alpha \mathds{1}_{\{ h_\eps>0 \}}(x)\\
V_t^{\delta,\eps} (x) = &(1- \delta) \sum_{z\in \eps\Z^d / \Z^d} H_\delta^\eps(z)\fint_{t-\delta}^{t+\delta} V_s(x-z)\,ds.
\end{aligned}
\end{equation}
Here $\rho_t^\eps, V_t^\eps$ are first extended for $t\in \R\setminus [0,1]$ constantly and by $0$ respectively, and $H_\delta^\eps$ is the discrete heat kernel for $\eps \Z^d / \Z^d$ at time $\delta$.

The averaged versions have the following properties:
\begin{equation}\label{eq: mollified regularity}
\begin{aligned}
|\rho_t^{\eps, \delta}(x+\eps e_i) - \rho_t^{\eps, \delta}(x)| \leq & C(\delta)\eps\\
\int_{\T^d} |V_t^{\eps, \delta}(x+\eps e_i) - V_t^{\eps, \delta}(x)|^2 \,dx \leq & C(\delta) \eps^2\\
\|\partial_t \rho_t^{\eps,\delta}\|_{L^\infty} \leq & C(\delta)\eps.
\end{aligned}
\end{equation}

We note that $\rho_t^{\eps,\delta}(dx) \leq h_\eps(x)dx$ and $\partial_t \rho_t + \divergence V_t^{\eps,\delta} = 0$ in $\D'((0,1)\times\T^d)$, and by the convexity of the function $(m,U)\mapsto \frac{|U|^2}{m}$, we have
\begin{equation}\label{eq: mollified energy}
\int_0^1 \int_{\T^d} \frac{|V_t^{\eps, \delta}|^2}{ \rho_t^{\eps,\delta}}\,dx\,dt \leq \int_0^1 \int_{\T^d} \frac{|V_t^\eps|^2}{ \rho_t^\eps}\,dx\,dt.
\end{equation}
\\
\emph{Step 2:}
Define for $x\in \T^d$ the cube $Q_{x,\eps} = (x + [0,\eps)^d )/ \Z^d \subset \T^d$.
Define $m_t^{\eps,\delta}(x) = \fint_{Q_{x,\eps}} \rho_t^{\eps,\delta}\,dy$ and $U_t^{\eps,\delta}(x) = \fint_{Q_{x,\eps}} V_t^{\eps,\delta}\,dy$.
Note that $\partial_t m_t^{\eps, \delta} + \divergence U_t^{\eps,\delta} = 0$ in $\D'((0,1)\times\T^d)$.
We now find a competitor for $f_{\hom}(m_t^{\eps,\delta}(x),U_t^{\eps,\delta}(x))$ for almost every $x,t$:

Consider the Hilbert space $H^1_{\eps\text{-per}}(\{h_\eps > 0\})$ of $\eps\Z^d$-periodic functions with mean $0$, equipped with symmetric bilinear form $\mathcal{A}(\phi,\psi) = \int_{Q_{x,\eps} \cap \{h_\eps>0\}} \nabla \phi \cdot \nabla \psi \,dy$,  which is independent of $x\in \T^d$ and positive definite by Lemma \ref{lma: poincare-trace}.

By the Lax-Milgram theorem, we may thus find for every $x\in \T^d$ a weak solution $\phi_{t,x}^{\eps,\delta} \in H^1_{\eps\text{-per}}(\{h_\eps > 0\})$ of
\begin{equation}\label{eq: weak equation}
\int_{Q_{x,\eps}\cap \{h_\eps > 0\}} \nabla \phi_{t,x}^{\eps,\delta} \cdot \nabla \psi \,dy = -\int_{Q_{x,\eps} \cap \{h_\eps > 0\}} V_t^{\eps,\delta} \cdot \nabla \psi \,dy
\end{equation}
for every $\psi \in H^1_{\eps\text{-per}}(\{h_\eps > 0\})$. Moreover, through integration by parts, using H\"older's inequality and Lemma \ref{lma: poincare-trace}, we may estimate
\begin{equation}
\begin{aligned}\label{eq: estimate Lax}
&-\int_{Q_{x,\eps}} V_t^{\eps,\delta} \cdot \nabla \psi \,dy\\
 = &\int_{\partial Q_{x,\eps}\cap \{h_\eps > 0\}} \frac{1}{2} (V_t^{\eps,\delta}(y) - V_t^{\eps,\delta}(y-\eps n(y))) \cdot n(y) \psi(y) \,d\Hm^{d-1}(y) - \int_{Q_{x,\eps} \cap \{h_\eps > 0\}} \divergence V_t^{\eps}(y) \psi(y)\,dy\\
\leq & C \eps \left(\|\divergence V_t^{\eps,\delta}\|_{L^2(Q_{x,\eps} \cap \{h_\eps > 0\})} + \|V_t^{\eps,\delta}(\cdot - \eps n) - V_t^{\eps,\delta} \|_{L^2(\partial Q_{x,\eps} \cap \{h_\eps > 0\})} \right) \|\nabla \psi\|_{L^2(Q_{x,\eps}\cap \{h_\eps > 0\})},
\end{aligned}
\end{equation}
where we used the fact that $\psi$ is $\eps \Z^d$-periodic and that $V_t^{\eps,\delta}\cdot n = 0$ on $\partial\{h_\eps > 0\}$ in the sense of distributions.

Inserting the solution $\phi_{t,x}^{\eps,\delta}$ into \eqref{eq: estimate Lax} and using the estimates in \eqref{eq: mollified regularity}, we find through Fubini's theorem and H\"older's inequality that for every $t\in [0,1]$ we have
\begin{equation}\label{eq: L^2 corrector estimate}
\int_{\T^d} \eps^{-d} \int_{Q_{x,\eps}\cap \{h_\eps > 0\}} |\nabla \phi_{t,x}^{\eps,\delta}(y)|^2\,dy \,dx \leq C(\delta) \eps^2.
\end{equation}

Further, the vector field $W_t^{\eps,\delta}=V_t^{\eps,\delta} +\nabla \phi_{t,x}^{\eps,\delta}\in L^2(Q_{x,\eps}\cap \{h_\eps>0\};\R^d)$ can be extended periodically to all of $\{h_\eps > 0\}$ and then by 0 in $\{h_\eps=0\}$, and the extension has zero distributional divergence in $\T^d$ by \eqref{eq: weak equation}. It follows that
\begin{equation}
f_{\hom}\left(m_t^{\eps, \delta}(x), U_t^{\eps,\delta}(x)+ \fint_{Q_{x,\eps}} \nabla \phi_{t,x}^{\eps,\delta}(y)\,dy \right) \leq \eps^{-d} \int_{Q_{x,\eps} \cap \{h_\eps > 0\}} \frac{|W_t^{\eps,\delta}(y)|^2}{\rho_t^{\eps,\delta}(y)}\,dy.
\end{equation}

We also use inequality $(vi)$ from Lemma \ref{lma: prophom} to obtain
\begin{equation}
\begin{aligned}
f_{\hom}\left(m_t^{\eps, \delta}(x), U_t^{\eps,\delta}(x)\right) \leq &f_{\hom}\left(m_t^{\eps, \delta}(x), U_t^{\eps,\delta}(x)+ \fint_{Q_{x,\eps}} \nabla \phi_{t,x}^{\eps,\delta}(y)\,dy \right) \\
& + C\frac{\left|U_t^{\eps,\delta}(x)\right|\left|\fint_{Q_{x,\eps}} \nabla \phi_{t,x}^{\eps,\delta}(y)\,dy \right|}{m_t^{\eps,\delta}(x)}.
\end{aligned}
\end{equation}

Integrating over $\T^d \times [0,1]$ yields
\begin{equation}
\begin{aligned}
& \int_0^1 \int_{\T^d} f_{\hom} (m_t^{\eps, \delta}(x), U_t^{\eps,\delta}(x))\,dx\,dt\\
\leq &  \int_0^1 \int_{\T^d} \eps^{-d}\int_{Q_{x,\eps}\cap \{h_\eps > 0\}}\frac{|V_t^{\eps,\delta}(y)|^2}{\rho_t^{\eps,\delta}(y)} + C\frac{|V_t^{\eps,\delta}(y) + \nabla \phi_{t,x}^{\eps,\delta}(y)||\nabla \phi_{t,x}^{\eps,\delta}(y)|}{\rho_t^{\eps,\delta}(y)}\,dy\,dx\,dt\\
 & + \underbrace{C(\delta) \int_0^1 \int_{\T^d} \left|U_t^{\eps,\delta}(x)\right|\left|\fint_{Q_{x,\eps}} \nabla \phi_{t,x}^{\eps,\delta}(y)\,dy \right| \,dx\,dt}_{I} \\
\leq & \int_0^1 \int_{\T^d} \frac{|V_t^{\eps,\delta}(x)|^2}{\rho_t^{\eps,\delta}(x)}\,dx\,dt + I \\
& + \underbrace{C(\delta)\eps^{-d} \int_0^1 \int_{\T^d} \int_{Q_{x,\eps} \cap \{h_\eps > 0\}} \left(|V_t^{\eps,\delta}(y) + \nabla \phi_{t,x}^{\eps,\delta}(y)|\right)|\nabla \phi_{t,x}^{\eps,\delta}(y)|\,dy\,dx\,dt}_{I\!I}.
\end{aligned}
\end{equation}
where we used Jensen's inequality and the convexity of the function $(m,U) \mapsto \frac{|U|^2}{m}$ for the first term, and the lower bound $\rho_t^{\eps,\delta} \geq \delta\alpha$ in $\{h_\eps > 0\}$ for the second.

We can then comfortably bound the error terms by repeatedly applying H\"older's inequality and \eqref{eq: L^2 corrector estimate}.
\begin{equation}
\begin{aligned}
I \leq &  C(\delta)\left(\int_0^1 \int_{\T^d} |U_t^{\eps,\delta}(x)|^2\,dx\,dt\right)^{1/2} \left( \int_0^1 \int_{\T^d} |\fint_{Q_{x,\eps}}\nabla \phi_{t,x}^{\eps,\delta}(y)\,dy|^2\,dx\,dt\right)^{1/2}\\
\leq & C(\delta)\left(\int_0^1 \int_{\T^d} |V_t^{\eps,\delta}(x)|^2\,dx\,dt\right)^{1/2} \left(\int_0^1 \int_{\T^d} \eps^{-d} \int_{Q_{x,\eps}}|\nabla \phi_{t,x}^{\eps,\delta}(y)|^2\,dy\,dx\,dt\right)^{1/2}\\
\leq &C(\delta) \eps.
\end{aligned}
\end{equation}
Similarly,
\begin{equation}
\begin{aligned}
 I\!I \leq & C(\delta)\eps^{-d} \int_0^1 \int_{\T^d} \left\|V_t^{\eps,\delta}(y) - \nabla \phi_{t,x}^{\eps,\delta}\right\|_{L^2(Q_{x,\eps} \cap \{h_\eps > 0\})}\|\nabla \phi_{t,x}^{\eps,\delta}\|_{L^2(Q_{x,\eps} \cap \{h_\eps > 0\})}\,dx\,dt\\
\leq & C(\delta)\left(\int_0^1 \int_{\T^d} |V_t^{\eps,\delta}(x)|^2 + \eps^{-d}\|\nabla \phi_{t,x}^{\eps,\delta}\|_{L^2(Q_{x,\eps} \cap \{h_\eps > 0\})}^2 \,dx\,dt\right)^{1/2}\\
&\times \left(\int_0^1 \int_{\T^d}\eps^{-d}\|\nabla \phi_{t,x}^{\eps,\delta}\|_{L^2(Q_{x,\eps} \cap \{h_\eps > 0\})}^2\,dx\,dt\right)^{1/2}\\
\leq &C(\delta)(\eps + \eps^2)
\end{aligned}
\end{equation}
Combine this with \eqref{eq: mollified energy} to obtain
 \begin{equation}
\liminf_{\eps \to 0}\int_0^1 \int_{\T^d} f_{\hom}(m_t^{\eps,\delta},U_t^{\eps,\delta})\,dx\,dt \leq \liminf_{\eps \to 0}\int_0^1 \int_{\T^d} \frac{|V_t^\eps|^2}{\rho_t^\eps}\,dx\,dt
\end{equation}
for every $\delta > 0$. Using a diagonal sequence $\delta(\eps) \to 0$, we see that $m_t^{\eps,\delta(\eps)} \weakstar \rho_t$ for almost every $t\in [0,1]$. The claim follows then from Lemma \ref{lma: lower semicontinuity}.
\end{proof}

\begin{proof}[Proof of the upper bound]
We have to show that for all curves of measures $(\rho_t)_{t\in[0,1]}\subset \M_+(\T^d)$ there exists for every $\eps=\frac1n$ a curve of measures
$(\rho_t^\eps)_{t\in[0,1]}$ such that as $\eps\to0$ we have for all $t\in[0,1]$  $\rho^\eps_t\weakstar\rho_t$ and
\begin{align}
 \limsup_{\eps\to0}E_{h_\eps}((\rho_t^\eps)_{t\in[0,1]})\leq E_{\hom}((\rho_t)_{t\in[0,1]}).
\end{align}
\\
\emph{Step 1:} We may assume that $(\rho_t)_{t\in[0,1]}$ has finite energy. We mollify 
in time and space with a standard mollifier. Let us call this curve $(\tilde \rho_t)_{t\in[0,1]}\in\C^\infty([0,1]\times \T^d)$ and the corresponding optimal
momentum vector field $(\tilde V_t)_{t\in[0,1]}\in\C^\infty([0,1]\times \T^d,\R^d)$.
\\
\emph{Step 2:} We fix a number $M\in \N$ of time steps satisfying $\eps\ll \frac1M\ll 1$. We define for $t_i\coloneqq\frac{i}{M}$ and $z\in (\eps \Z^d)/ \Z^d$ the following objects
\begin{equation}
\begin{aligned}
 m_{t_i}(z)&\coloneqq\fint_{Q(z,\eps)}\tilde \rho_{t_i}\ dx\\
 U_{t_i}(z,z\pm\eps e_j)&\coloneqq\fint_{t_{i-1}}^{t_{i+1}}\fint_{\partial Q(z,\eps)\cap\partial Q(z\pm\eps e_j,\eps)}\tilde V_s\cdot (\pm e_j)\, d\Hm^{d-1}(x)\, ds.
\end{aligned}
\end{equation}
Note that for $t\in(t_i,t_{i+1})$ with $m_t$ the linear interpolation between $m_{t_i}$ and $m_{t_{i+1}}$
\begin{align}\label{eq: discrete continuity}
\partial_tm_t(z)+\eps^{-1}\sum_{j=1}^d(U_{t_i}(z,z+\eps e_j)+U_{t_i}(z,z-\eps e_j))=0.
\end{align}
\\
\emph{Step 3:} We insert the optimal microstructures $\nu_{t_i,z}\in L^1(\{h>0\})$ and $W_{t_i,z}\in L^2(\{h>0\};\R^d)$ for $f_{\hom}(m_{t_i}(z),U_{t_i}(z))$, where
\begin{align}
U_{t_i}(z)\cdot e_j \coloneqq U_{t_i}(z,z+\eps e_j).
\end{align}
Fix $a\in [0,\eps)^d/\Z^d$ to be chosen later, and define for $x\in Q_{z+a,\eps}$
\begin{equation}\label{eq: average upper bound}
\begin{aligned}
 \rho_{t_i}^\eps(x) \coloneqq &(1-\delta)\sum_{z'\in (\eps \Z^d)/ \Z^d}H_\delta^\eps(z-z')\nu_{t,z'}(\frac{x}{\eps})+\delta\alpha \mathds{1}_{\{h_\eps>0\}}(x),\\
 V_{t_i}^\eps(x) \coloneqq &(1-\delta)\sum_{z'\in(\eps \Z^d)/ \Z^d}H_\delta^\eps(z-z')W_{t,z'}(\frac{x}{\eps}),\\
 X_{t_i}^\eps(x) \coloneqq &(1-\delta) \sum_{z' \in (\eps \Z^d) / \Z^d} H_\delta^\eps(z-z') X_{U_{t_i}(z') - U_{t_{i+1}}(z')}(\frac{x}{\eps}),
\end{aligned}
\end{equation}
where $\delta>0$, $\alpha$ is the positive lower bound of the function $h$, $H_\delta^\eps$ is the discrete heat flow on 
$(\eps \Z^d)/ \Z^d$, and $X_U\in L^2(\{h>0\};\R^d)$ is the vector field from Lemma \ref{lma: potential flow}.
\\
\emph{Step 4:} For $t\in (t_i,t_{i+1})$ we define $\rho^\eps_t\in L^\infty(\T^d)$ and $ V^\eps_t\in L^2(\{h_\eps>0\};\R^d)$ as the linear interpolations
\begin{equation}
\begin{aligned}
 \rho_t^\eps(x)\coloneqq&\frac{t_{i+1}-t}{t_{i+1}-t_i}\rho_{t_i}^\eps(x)+\frac{t-t_i}{t_{i+1}-t_i}\rho_{t_{i+1}}^\eps(x),\\
 V_t^\eps(x)\coloneqq&\frac{t_{i+1}-t}{t_{i+1}-t_i} V_{t_i}^\eps(x)+\frac{t-t_i}{t_{i+1}-t_i} \left(V_{t_{i+1}}^\eps(x) + X_{t_i}^\eps(x) \right).
\end{aligned}
\end{equation}
We see that
\begin{equation}
\divergence V_t^\eps = \sum_{z\in (\eps \Z^d)/\Z^d} \sum_{j=1}^d - [V_t^\eps]\cdot e_j \Hm^{d-1}|_{\partial Q_{z+a,\eps} \cap \partial Q_{z + a - \eps e_j, \eps}}
\end{equation}
where $[V_t^\eps]$ denotes the jump of $V_t^\eps$ from $Q_{z+a,\eps}$ to $Q_{z+a-\eps e_j,\eps}$. Note that since $V_t^\eps\in L^2(\T^d;\R^d)$, by Fubini's theorem the above is defined for almost every $a\in [0,\eps)^d/\Z^d$.

 Moreover, for every $z\in (\eps\Z^d)/\Z^d$ we have
\begin{equation}
\int_{Q_{z+a,\eps}} V_t^\eps (x) \,dx = (1-\delta) \sum_{z'\in (\eps\Z^d)/\Z^d} H_\delta^\eps(z-z') U_{t_i}(z'),
\end{equation}
and
\begin{equation}
\int_{Q_{z+a,\eps}} \partial_t\rho_t^\eps(x)\,dx = (1-\delta)\sum_{z'\in (\eps\Z^d)/\Z^d} H_\delta^\eps(z-z') \frac{m_{t_{i+1}}(z') - m_{t_i}(z')}{t_{i+1}-t_i}.
\end{equation}

Combining the above with \eqref{eq: discrete continuity}, we also obtain that
\begin{equation}
\int_{Q_{z+a,\eps}} \partial_t\rho_t^\eps(x)\,dx + \sum_{j=1}^d  \int_{\partial Q_{z+a,\eps} \cap \partial Q_{z + a - \eps e_j, \eps}}- [V_t^\eps]\cdot e_j \,d\Hm^{d-1} = 0,
\end{equation}
or more concisely
\begin{equation}\label{eq: cube}
(\partial_t \rho_t^\eps + \divergence V_t^\eps) (Q_{z+a,\eps}) = 0,
\end{equation}
for every $z\in (\eps\Z^d)/\Z^d$, for almost every $a$. Note that in \eqref{eq: cube} it is imperative that $Q_{z+a\eps}$ be the \emph{half-open} cubes.
\\
\emph{Step 5:} Let $\phi_t^\eps\in H^1(\{h_\eps > 0\})$ be the weak solution to
\begin{equation}
\begin{cases}
\Delta \phi_t^\eps = -(\partial_t \rho_t^\eps + \divergence V_t^\eps)&\text{, in }\{h_\eps > 0\}\\
\nabla \phi_t^\eps \cdot n = 0 &\text{, on }\partial \{h_\eps>0\}\\
\int_{\{h_\eps > 0\}} \phi_t^\eps(x)\,dx = 0,
\end{cases}
\end{equation}
i.e. the unique function in the Hilbert space $H^1_\eps \coloneqq\{\psi\in H^1(\{h_\eps > 0\})\,:\,\int_{\{h_\eps > 0\}} \psi(x)\,dx = 0\}$ with
\begin{equation}
\begin{aligned}\label{eq: weak form}
&\int_{\{h_\eps > 0\}} \nabla \phi_t^\eps \cdot \nabla \psi \,dx \\
&= \sum_{z\in (\eps \Z^d)/\Z^d} \int_{Q_{z+a,\eps}\cap \{h_\eps > 0\}} \partial_t\rho_t^\eps \psi\,dx + \sum_{j=1}^d  \int_{\partial Q_{z+a,\eps} \cap \partial Q_{z + a - \eps e_j, \eps} \cap \{h_\eps > 0\}}- [V_t^\eps]\cdot e_j \psi \,d\Hm^{d-1}
\end{aligned}
\end{equation}
for all $\psi \in H_\eps^1$. Note that after extending $\nabla\phi_t^\eps$ by $0$ in $\{h_\eps=0\}$,  \eqref{eq: weak form} actually holds for all $\psi\in H^1(\T^d)$.

By the Lax-Milgram Theorem, a unique such $\phi_t^\eps$ exists for almost every $a$. Testing with $\psi = \phi_t^\eps$ and using Lemma \ref{lma: poincare-trace}, we see that
\begin{equation}
\begin{aligned}
\int_{\{h_\eps > 0 \}} |\nabla \phi_t^\eps|^2\,dx = &  \sum_{z\in (\eps \Z^d)/\Z^d} \bigg( \int_{Q_{z+a,\eps}} \partial_t\rho_t^\eps (\phi_t^\eps - \overline{\phi_t^\eps}_z)\,dx\\
& + \sum_{j=1}^d  \int_{\partial Q_{z+a,\eps} \cap \partial Q_{z + a - \eps e_j, \eps}}- [V_t^\eps]\cdot e_j (\phi_t^\eps - \overline{\phi_t^\eps}_z) \,d\Hm^{d-1}\bigg)\\
\leq & C(\{h_\eps > 0\})\eps \left(\|\partial_t \rho_t^\eps\|_{L^2(\{h_\eps > 0\})} + \frac{1}{\sqrt{\eps}}\|[V_t^\eps]\|_{L^2(\bigcup_z \partial Q_{z+a,\eps})}\right) \|\nabla \phi_t^\eps\|_{L^2(\{h_\eps > 0\})}. 
\end{aligned}
\end{equation}

At this point we pick $a\in [0,\eps)^2/\Z^d$ such that $\|[V_t^\eps]\|_{L^2(\bigcup_z \partial Q_{z+a,\eps})}^2 \leq C(\delta) \eps$, which is possible by Fubini's theorem and the regularity of the discrete heat flow. Of course, $\|\partial_t \rho_t^\eps\|_{L^\infty} \leq CM$, so that
\begin{equation}\label{eq: hnorm}
\int_{\{h_\eps > 0\}} |\nabla \phi_t^\eps|^2\,dx \leq (C(\delta) + CM^2)\eps^2.
\end{equation}

Further, taking $W_t^\eps = V_t^\eps + \nabla \phi_t^\eps \mathds{1}_{\{h_\eps > 0\}}$, we see by \eqref{eq: weak form} that $\partial_t \rho_t^\eps + \divergence W_t^\eps = 0$ in $\D'((0,1)\times\T^d)$.
\\
\emph{Step 6:} We estimate using \eqref{eq: hnorm} that
\begin{equation}\label{eq: first upper}
\begin{aligned}
 \int_{\T^d}\frac{|W^\eps_t|^2}{\rho^\eps_t}\,dx\leq &\int_{\T^d}\frac{ |V_t^\eps|^2}{\rho_t^\eps}\, dx+\frac{C}{\delta\alpha}\int_{\T^d} |\nabla \phi_t^\eps||V_t^\eps + \nabla \phi_t^\eps|\,dx\\
 \leq & \int_{\T^d}\frac{ |V_t^\eps|^2}{\rho_t^\eps}\, dx + \frac{C(\delta)\eps}{\alpha}(1+\eps)(1 + M)\\
 \leq & \int_{\T^d}\frac{ |V_t^\eps|^2}{\rho_t^\eps}\, dx + \frac{C(\delta)}{\alpha}M\eps.
\end{aligned}
\end{equation}
Using the joint convexity of the function $(m,U)\mapsto \frac{|U|^2}{m}$ we find for $t\in (t_i,t_{i+1})$ that
\begin{equation}\label{eq: second upper}
\begin{aligned}
 \int_{\T^d}\frac{|V^\eps_t|^2}{\rho^\eps_t}\,dx\leq&\frac{t_{i+1}-t}{t_{i+1}-t_i}\int_{\T^d}\frac{|V^\eps_{t_i}|^2}{\rho^\eps_{t_i}}\,dx+\frac{t-t_i}{t_{i+1}-t_i}\int_{\T^d}\frac{|V^\eps_{t_{i+1}}+X^\eps_{t_i}|^2}{\rho^\eps_{t_{i+1}}}\, dx\\
 \leq &\frac{t_{i+1}-t}{t_{i+1}-t_i}\int_{\T^d}\frac{|V^\eps_{t_i}|^2}{\rho^\eps_{t_i}}+\frac{t-t_i}{t_{i+1}-t_i}\int_{\T^d}\frac{|V^\eps_{t_{i+1}}|^2}{\rho^\eps_{t_{i+1}}}\, dx\\
 &+\underbrace{\frac{C(\delta)}{\alpha}\int_{\T^d}|X^\eps_{t_i}||X^\eps_{t_i} + V_t^\eps|\, dx}_{I}.
 \end{aligned}
 \end{equation}
 We estimate $I$ using Lemma \ref{lma: potential flow}:
 \begin{equation}\label{eq: third upper}
 \begin{aligned}
  I \leq &\frac{C(\delta)}{\alpha}\|X^\eps_{t_i}\|_{L^2}(\|X^\eps_{t_i}\|_{L^2} + \|V_t^\eps\|_{L^2}) \\
  \leq &\frac{C(\delta)}{\alpha}\left(\sum_{z'\in(\eps\Z^d)/\Z^d}\eps^d|U_{t_i}(z')-U_{t_{i+1}}(z')|^2\right)^{1/2}\leq \frac{C(\delta)}{\alpha M}.
 \end{aligned}
 \end{equation}
 
 For the main term we find through exploiting the convexity and the definition of $V_{t_i}^\eps,\rho_{t_i}^\eps$ that
 \begin{equation}\label{eq: fourth upper}
 \begin{aligned} \int_{\T^d}\frac{|V^\eps_{t_i}|^2}{\rho^\eps_{t_i}}\,dx \leq \sum_{z\in(\eps\Z^d)/\Z^d}\eps^d f_{\hom}(m_{t_i}(z),U_{t_i}(z)).
 \end{aligned}
 \end{equation}

 We now combine the estimates \eqref{eq: first upper}, \eqref{eq: second upper}, \eqref{eq: third upper}, \eqref{eq: fourth upper} and integrate in time so that
 \begin{equation}
 \begin{aligned}
 \int_0^1 \int_{\T^d} \frac{|W_t^\eps|^2}{\rho_t^\eps}\,dx\,dt \leq & \sum_{i=0}^M \sum_{z\in (\eps \Z^d)/\Z^d} \frac{\eps^d}{M}f_{\hom}(m_{t_i}(z), U_{t_i}(z)) + \frac{C(\delta)}{\alpha}\left(\frac{1}{M} + M\eps\right). 
 \end{aligned}
 \end{equation}
 
 Finally, we use the Lipschitz continuity of $f_{\hom}$ from Lemma \ref{lma: prophom} and the Lipschitz continuity of $(\tilde \rho, \tilde V)$ to estimate the Riemann sum above by the integral
 \begin{equation}
 \sum_{i=0}^M \sum_{z\in (\eps \Z^d)/\Z^d} \frac{\eps^d}{M}f_{\hom}(m_{t_i}(z), U_{t_i}(z)) \leq \int_0^1 \int_{\T^d} f_{\hom}(\tilde \rho_t(x),\tilde V_t(x))\,dx\,dt + C(\delta) (\frac{1}{M} + \eps).
 \end{equation}

Choosing $M=\lfloor{\eps^{-1/2}}\rfloor$ and letting $\eps\to 0$ we otain the desired estimate
\begin{align}
 \limsup_{\eps\to0}E_{h_\eps}((\rho^\eps_t)_{t\in[0,1]})\leq\int_0^1\int_{\T^d} f_{\hom}({d\tilde\rho_t},\tilde V_t)\, dx\, dt\leq E_{\hom}((\rho_t)_{t\in[0,1]}),
\end{align}
where we used the convexity of $E_{\hom}$ in the last equality (Lemma \ref{lma: prophom}).
 Finally take a diagonal sequence such that $\rho_t^\eps\weakstar\rho_t$ for all $t\in[0,1]$.
\end{proof}

\begin{rem}
Finally, we note that we may add lower bounds on the density, in the form $\rho_t(A) \geq \int_A l(x)\,dx$ for every closed set $A$, with a measurable lower density bound $l\in L^1(\Omega)$, with $l(x)\leq h(x)$ almost everywhere. This is just another convex constraint.

In fact, Theorem \ref{thm:homogenizationintro} can be proved under the additional constraint for $l^\eps(x) = l(x/\eps)$ with a few easy modifications. In \eqref{eq: f_hom}, we take the infimum with the additional constraint that $\nu(x)\geq l(x)$ almost everywhere, increasing the energy.

In Lemma \ref{lma: prophom}, the upper bound in (iv) then has to be replaced by
\begin{equation}
f_{\hom}(m,U) \leq C\frac{|U|^2}{m - \int_{\T^d}l(x)\,dx}.
\end{equation}
\end{rem}

Finally, in \eqref{eq: average lower bound} and \eqref{eq: average upper bound}, the term $\delta \alpha \mathds{1}_{\{h_\eps > 0\}}$ has to be replaced by $\delta (\alpha \mathds{1}_{\{h_\eps > 0\}} \lor l_\eps)$.

\bibliography{obstacle_transport}
    
\end{document}